\documentclass[3p,12pt]{elsarticle}

\usepackage{amsmath,amssymb,amsfonts,amsthm}
\usepackage[strings]{underscore}
\usepackage{booktabs}
\usepackage{hyperref}
\usepackage[noabbrev]{cleveref}

\usepackage{ifthen}
\newboolean{UseExternalPdfs}
\setboolean{UseExternalPdfs}{true}
\usepackage{algorithm,algorithmic}
\usepackage{xcolor}
\DeclareMathOperator{\spn}{span}

\newdefinition{problem}{Problem}
\crefname{problem}{Problem}{Problems}
\newtheorem{theorem}[problem]{Theorem}
\crefname{theorem}{Theorem}{Theorems}
\newtheorem{lemma}[problem]{Theorem}
\crefname{lemma}{Lemma}{Lemmas} 
\newtheorem{definition}[problem]{Definition}
\crefname{definition}{Definition}{Definitions}
\usepackage{epstopdf}
\usepackage{algorithmic}
\ifpdf
  \DeclareGraphicsExtensions{.eps,.pdf,.png,.jpg}
\else
  \DeclareGraphicsExtensions{.eps}
\fi

\numberwithin{problem}{section}

\journal{Computer Methods in Applied Mechanics and Engineering}

\begin{document}

\begin{frontmatter}



\title{Stochastic Galerkin reduced basis methods for parametrized linear elliptic partial differential equations}

\author[1]{Sebastian Ullmann\fnref{fn1}}
\ead{ullmann@mathematik.tu-darmstadt.de}

\author[2]{Christopher M\"uller\corref{cor1}}
\ead{cmueller@mathematik.tu-darmstadt.de}

\author[2]{Jens Lang\fnref{fn2}}
\ead{lang@mathematik.tu-darmstadt.de}

\cortext[cor1]{Corresponding author}
\fntext[fn1]{Sebastian Ullmann was supported by the German Research Foundation within the Graduate School of Excellence 
Computational Engineering (DFG GSC233).}
\fntext[fn2]{Jens Lang is supported by the German Research Foundation within the collaborative research center TRR154 
“Mathematical 
Modeling, Simulation and Optimisation Using the Example of Gas Networks” (Project-ID 239904186, TRR154/2-2018, 
TP B01).}
\address[1]{Technical University of Darmstadt, Graduate School of Excellence Computational Engineering, Dolivostra\ss e 
15, 64293 Darmstadt, Germany}
\address[2]{Technical University of Darmstadt, Department of Mathematics, Dolivostra\ss e 15, 64293 Darmstadt, 
Germany}

\begin{abstract}
We consider the estimation of parameter-dependent statistics of functional outputs of elliptic boundary value problems 
(BVPs) with parametrized random and deterministic inputs. For a given value of the deterministic paremeter, a stochastic 
Galerkin finite element (SGFE) method can estimate the corresponding expectation and variance of a linear output at the 
cost of a single solution of a large block-structured linear system of equations. We propose a stochastic Galerkin 
reduced basis (SGRB) method as a means to lower the computational burden when statistical outputs are required for a 
large number of deterministic parameter queries.~Our working assumption is that we have access to the 
computational resources necessary to set up such a reduced order model for a spatial-stochastic weak formulation of the 
parameter-dependent BVP. In this scenario, the complexity of evaluating the SGRB model for a new value of the 
deterministic parameter only depends on the reduced dimension. 

To derive an SGRB model, we project the spatial-stochastic weak solution of a parameter-dependent SGFE model onto a POD 
reduced basis generated from snapshots of SGFE solutions at representative values of the parameter.~We propose 
residual-corrected estimates of the parameter-dependent expectation 
and variance of linear functional outputs and provide respective computable error bounds.~We test the SGRB method 
numerically for a convection-diffusion-reaction problem, choosing the convective velocity as a deterministic parameter 
and the parametrized reactivity field as a random input. Compared to a standard reduced basis model embedded in a Monte 
Carlo sampling procedure, the SGRB model requires a similar number of reduced basis functions to meet a given tolerance 
requirement. However, only a single run of the SGRB model suffices to estimate a statistical output for a new 
deterministic parameter value, while the standard reduced basis model must be solved for each Monte Carlo sample.

\end{abstract}


\begin{highlights}
\item ROMs to estimate parametrized statistics of functional outputs of elliptic PDEs.
\item Residual-corrected estimation of expectation and variance of statistical outputs.
\item Evaluation of stochastic Galerkin reduced basis (SGRB) model is sampling-free.
\item SGRB method achieves similar reduction of degrees of freedom as Monte Carlo RB.
\end{highlights}


\end{frontmatter}


\section{Introduction}

\newcommand{\sampleSpace}{\Theta}
\newcommand{\sigmaAlgebra}{\mathcal F}
\newcommand{\probability}[1][]{\mathbb P_{#1}}
\newcommand{\randomVariable}[1][]{\xi_{#1}}
\newcommand{\density}[1][]{p_{#1}}
\newcommand{\E}[2][]{\mathbb E#1[#2#1]}
\newcommand{\V}[2][]{\mathbb V#1[#2#1]}

\newcommand{\dd}{\,\mathrm d}
\newcommand{\RR}{\mathbb R}
\newcommand{\NN}{\mathbb N}

\newcommand{\rand}[1][]{y_{#1}}
\newcommand{\iRand}{k}
\newcommand{\nRand}{K}
\newcommand{\dRand}[1][]{\Xi_{#1}}

\newcommand{\lebesgue}[3][2]{L^{#1}_{#2}(#3)}
\newcommand{\bochner}[3][]{\lebesgue{#1}{#2,#3}}
\newcommand{\banach}[4][]{\lebesgue[#1]{#2}{#3,#4}}

\newcommand{\hilbSPF}{X}
\newcommand{\hilbSDF}{\hilbSPF'}

\newcommand{\hilbSg}{S}
\newcommand{\hilbSgRed}{\hilbSg_{\nRb}}

\newcommand{\iSnap}{n}
\newcommand{\nSnap}{N}

\newcommand{\randVar}[1]{\randomVariable^{#1}}
\newcommand{\iRandVar}{\iSnap}
\newcommand{\nRandVar}{\nSnap_{\randomVariable}}
\newcommand{\dRandVar}{\dRand^{\nRandSnap}}

\newcommand{\randSnap}[1]{\rand^{#1}}
\newcommand{\iRandSnap}{\iSnap}
\newcommand{\nRandSnap}{\nSnap_{\randomVariable}}
\newcommand{\dRandSnap}{\dRand^{\nRandSnap}}

\newcommand{\para}[1][]{\mu_{#1}}
\newcommand{\iPara}{p}
\newcommand{\nPara}{P}
\newcommand{\dPara}{\mathcal P}

\newcommand{\paraSnap}[1]{\para^{#1}}
\newcommand{\iParaSnap}{\iSnap}
\newcommand{\nParaSnap}{\nSnap_{\para}^\text{train}}
\newcommand{\dParaSnap}{\dPara^{\nParaSnap}}

\newcommand{\nRb}{R}
\newcommand{\hilbSPR}[1][\nRb]{\hilbSPF^#1}

\newcommand{\uSPF}{u}
\newcommand{\uSPFpara}[1][\rand,\para]{\uSPF(#1)}
\newcommand{\uSPR}{\uSPF^{\nRb}}
\newcommand{\uSPRpara}[1][\rand,\para]{\uSPR(#1)}

\newcommand{\uWPF}{\bar u}
\newcommand{\uWPFpara}[1][\para]{\uWPF(#1)}
\newcommand{\uWPFparai}[2][\para]{\uWPF^{#2}(#1)}
\newcommand{\uWPR}{\uWPF^{\nRb}}
\newcommand{\uWPRpara}[1][\para]{\uWPR(#1)}

We consider linear elliptic boundary-value problems subject to a finite number of random and deterministic input parameters. Our goal is to compute the parameter-dependent expected value and variance of a functional output of interest. In this context, a reduced basis model provides a computationally inexpensive map between the deterministic input parameter and the corresponding output statistics. Moreover, it provides a computable a posteriori bound for the error between the reduced basis statistical estimate and a corresponding high-fidelity estimate. Reduced basis methods for linear elliptic boundary value problems with affinely parametrized \emph{deterministic} data are well-understood \cite{HesthavenEA2016,PrudhommeEA2002,QuarteroniEA2016}. We consider two approaches to include \emph{stochastic} parameters:
\begin{itemize}
  \item \emph{Monte Carlo reduced basis (MCRB) method:} The underlying equations are formulated weakly regarding the physical space, which means that the problem depends on both the deterministic and stochastic parameters.~Monte Carlo sampling is used to estimate the parameter-dependent expected value and variance of a functional output of interest. An MCRB method for linear elliptic problems with error bounds for the expectation and variance of a linear functional output is derived in \cite{BoyavalEA2009}. Improved error bounds are provided by \cite{HaasdonkEA2013}. Further advances are the introduction of a weighted error estimator \cite{ChenEA2013} and the embedding in a multi-level procedure \cite{Vidal-CodinaEA2016}. MCRB methods have also been applied to parabolic problems \cite{SpannringEA2018}, saddle point problems \cite{NewsumPowell2017}, Bayesian inverse problems \cite{Boyaval2012,ChenSchwab2016,ManzoniEA2016} and the assessment of rare events \cite{ChenQuarteroni2013}.
  \item \emph{Stochastic Galerkin reduced basis (SGRB) method:} The underlying equations are formulated weakly 
			  regarding the spatial and stochastic dimensions, so that the problem depends on the deterministic 
parameters only. Parameter-dependent estimates of the expected value and variance of a functional output are obtained 
by direct integration of the reduced solution.~The principle of SGRB methods is introduced in \cite{VenturiEA2008} for 
stochastic time-dependent incompressible Navier-Stokes problems, formulated weakly regarding the spatial and stochastic 
dimensions, with time acting as a parameter. Applications to linear dynamical systems are studied in 
\cite{FreitasEA2016,PulchMaten2015}. SGRB methods can be related to space-time reduced basis methods 
\cite{GlasEA2017,UrbanPatera2012}, which rely on a weak formulation with respect to space and time. The idea of using 
SGRB methods to estimate parameter-dependent expected values is discussed in \citep[section 8.2.1]{Wieland2013}.
\end{itemize}
Our main contribution is the derivation of a stochastic Galerkin reduced basis method to compute residual-corrected 
parameter-dependent estimates of the expectation and variance of statistical outputs together with corresponding error 
bounds.~A Monte Carlo reduced basis method will be used as a benchmark to assess the accuracy of the SGRB method and to 
compare the error bounds.

The creation of an SGRB model requires an underlying stochastic Galerkin finite element 
(SGFE) model or an equivalent high-fidelity Galerkin approximation. At least a few snapshots of the SGFE solution for 
different values of the deterministic parameter are needed to provide a suitable reduced basis. Moreover, evaluations of 
the SGFE linear and bilinear forms are necessary to derive the respective reduced-order Galerkin model and error bounds. 
We assume that the resources necessary for these computations are available in the setup phase.~We do not discuss the 
associated costs explicitly as it is not the focus of this work but refer to \cite{DexterEtAl2016} for a detailed 
analysis of the costs of stochatic Galerkin methods. We point out that the SGRB method presented in this paper is 
not a tool to reduce the computational burden associated with a single solution of an SGFE model as it is the objective 
in, e.g., \cite{Nouy2007,TamelliniEA2014} using proper generalized decomposition, in \cite{PowellEA2017,LeeElman2017} 
using a rational Krylov method and a low-rank tensor approximation, respectively, and in 
\cite{PowellElman2009,MuellerEtAl2019} using problem-tailored preconditioned iterative solvers. Instead, the SGRB 
approach targets the situation where a certain number of SGFE simulations are feasible in an expensive pre-processing 
step to create a reduced-order model (ROM) which can be evaluated cheaply for any given deterministic 
parameter. As an extreme scenario, one could imagine having supercomputer resources available in the 
setup phase whereas the ROM shall be evaluated on a microcontroller in real time. Therefore, SGRB models can be 
particularly useful in settings where statistical estimates are required for many values of the deterministic parameter, 
like in robust optimal control or the real-time exploration of parameter-dependent statistics. 

Compared to Monte Carlo reduced basis methods, stochastic Galerkin reduced basis methods can substantially decrease the 
computational cost of estimating the expectation and variance for a given deterministic parameter value. The reason is 
that MCRB methods require sampling the reduced-order solution, which may lead to a large number of reduced-order 
simulations for a single query of the deterministic parameter. The same issue arises when a stochastic 
collocation method is applied instead of Monte Carlo \cite{ElmanLiao2013}. SGRB methods overcome this drawback by 
evaluating the stochastic integrals in the offline stage, i.e., during the setup of the ROM. As a 
result, the cost of solving an SGRB model is similar to the cost of a single solution of a comparable RB model within an 
MC loop. At the same time, the SGRB model directly delivers a statistical estimate without sampling. Therefore, one can 
expect a speed-up factor in the order of magnitude of the number of MC samples.

\section{Monte Carlo reduced basis method} \label{sec:mcrb}

\newcommand{\MC}[2][]{E#1[#2#1]}
\newcommand{\MCV}[2][]{V#1[#2#1]}
\newcommand{\MCM}[2][]{\underline E#1[#2#1]}

We introduce a complete probability space $(\sampleSpace,\sigmaAlgebra,\probability)$ consisting of a set $\sampleSpace$ of elementary events, a $\sigma$-algebra $\sigmaAlgebra$ on $\sampleSpace$ and a probability measure $\probability$ on $\sigmaAlgebra$. For $\iRand=1,\dots,\nRand$ with $\nRand\in\NN$, we define independent random variables $\randomVariable[\iRand]\colon \sampleSpace\rightarrow \dRand[\iRand]$, where $\dRand[\iRand]\subset\RR$ is the image of $\randomVariable[\iRand]$. We introduce respective probability distributions  $\probability[{\randomVariable[\iRand]}]$ and probability densities $\density[{\randomVariable[\iRand]}]:\dRand[\iRand]\rightarrow\RR^+$, so that $\probability[{\randomVariable[\iRand]}](B) = \int_B\density[{\randomVariable[\iRand]}](\rand)\dd \rand = \probability(\randomVariable[\iRand]^{-1}(B))$ for all $B$ in the Borel $\sigma$-algebra of $\dRand[\iRand]$. We collect the random variables in a random vector  $\randomVariable\colon\sampleSpace\rightarrow\dRand$, where $\randomVariable=(\randomVariable[1],\dots,\randomVariable[\nRand])^T$ and $\dRand=\dRand[1] \times \dots \times \dRand[\nRand]$, with joint distribution $\probability[\randomVariable]$ and density $\density[{\randomVariable}]:\dRand\rightarrow\RR^+$. We denote the expectation of any $\probability[\randomVariable]$-measurable function $g\colon\dRand\rightarrow\RR$ with density $\density[{\randomVariable}]$ by $\E{g}=\int_{\dRand} g(\rand)\dd\probability[\randomVariable](\rand)=\int_{\dRand} g(\rand)\density[{\randomVariable}](\rand)\dd \rand$. We define the variance $\V{g}=\E{(g-\E{g})^2}$ for any $g\in\lebesgue{\randomVariable}{\dRand}$, where $\lebesgue{\randomVariable}{\dRand}:=\{\,v:\dRand\rightarrow\RR\mid\int_{\dRand} v(\rand)^2\density[\randomVariable](\rand) \dd\rand<\infty\,\}$.

We introduce a deterministic parameter $\para\in\dPara$ for some parameter domain $\dPara\subset\RR^{\nPara}$ with $\nPara\in\NN$. The final statistical outputs are scalar-valued $\para$-dependent functions representing approximations to the expectation and variance of a linear functional of a PDE solution.

We let $\{\randVar{1},\dots,\randVar{\nRandSnap}\}$ be a set of independent copies of the random vector $\randomVariable$. For some $g\in\lebesgue{\randomVariable}{\dRand}$, we define Monte Carlo estimators
\begin{equation} \label{eq:MonteCarlo}
  \MC{g} := \frac1\nRandSnap\sum_{\iRandSnap=1}^{\nRandSnap}g(\randVar{\iRandSnap}),\quad
  \MCM{g} := \frac1{\nRandSnap-1}\sum_{\iRandSnap=1}^{\nRandSnap}g(\randVar{\iRandSnap}),\quad
  \MCV{g} := \MCM{g^2}-\MCM{g}\MC{g},
\end{equation}
for which $\E{\MC{g}}=\E{g}$ and $\E{\MCV{g}}=\V{g}$ hold. We let $\dRandSnap:=\{\randSnap{1},\dots,\randSnap{\nRandSnap}\}$ be a realization of $\{\randVar{1},\dots,\randVar{\nRandSnap}\}$. A realization of a Monte Carlo estimate is obtained after substituting $\randVar{\iRandSnap}$ by $\randSnap{\iRandSnap}$ in \eqref{eq:MonteCarlo}, assuming that $g(\rand)$ is computable for any $\rand\in\dRandSnap$. In our approach, we view $\nRandSnap$ as a discretization parameter and fix $\dRandSnap$ before we build the reduced basis model.

We focus on the case where $g$ depends on its argument via a discretized PDE problem. In the following, we provide a full-order model (\cref{sec:mcrbFeModel}) and a reduced-order model (\cref{sec:mcrbRbModel}) to approximate the solution of the PDE for a given realization of the deterministic and random input parameters. The computation of linear outputs and the corresponding statistics are described in \cref{sec:mcrbEstimates}, together with the respective error bounds. For the separation of the computation into an expensive offline phase and a inexpensive online phase, we refer to \cite{BoyavalEA2009,HaasdonkEA2013}.

\subsection{Monte Carlo finite element model} \label{sec:mcrbFeModel}

\newcommand{\nFe}{M_\text{FE}}
\newcommand{\iFe}{m}

\newcommand{\uTest}{v}
\newcommand{\uTrial}{w}

\newcommand{\aS}{a}
\newcommand{\aSpara}[3][\rand,\para]{\aS(#2,#3;#1)}

\newcommand{\fS}{f}
\newcommand{\fSpara}[2][\rand,\para]{\fS(#2;#1)}

\newcommand{\lSP}{l}
\newcommand{\lSPpara}[2][\rand,\para]{\lSP(#2;#1)}
\newcommand{\lSD}[1]{\lSP_{(#1)}}

We use a stochastic strong form of a pa\-ra\-me\-tri\-zed PDE problem with random data to formulate an MCFE model. Samples of the solution of the PDE problem are characterized by a separable Hilbert space $\hilbSPF$ with inner product $(\cdot,\cdot)_\hilbSPF$ and norm $\|\cdot\|_\hilbSPF$. We introduce a parametrized bilinear form $\aSpara{\cdot}{\cdot}:\hilbSPF\times\hilbSPF\rightarrow\RR$ as well as a parametrized linear form $\fSpara{\cdot}:\hilbSPF\rightarrow\RR$. This allows a stochastic strong formulation of a linear elliptic PDE problem:
\begin{problem}[MCFE model]\label{problem:strong}
For given $(\rand,\para)\in\dRandSnap\times\dPara$, find
\begin{equation}
  \uSPFpara\in\hilbSPF\;\colon\quad\aSpara{\uSPFpara}{\uTest} = \fSpara{\uTest}\quad\forall \uTest\in \hilbSPF.\label{eq:strongPrimal}
\end{equation}
\end{problem}
We assume that $\aSpara{\cdot}{\cdot}$ is $(\rand,\para)$-uniformly bounded and coercive on $\hilbSPF$ and that $\fSpara{\cdot}$ is $(\rand,\para)$-uniformly bounded on $\hilbSPF$. Then \cref{problem:strong} has a unique solution for any given $(\rand,\para)\in\dRandSnap\times\dPara$ according to the Lax-Milgram lemma.

It is a usual premise in reduced basis methods that the discretization space of the underlying full-order model is 
assumed to be large enough to capture the solution with sufficient precision. The assessment of the error of the 
full-order solution with respect to the infinite-dimensional exact solution is delegated to the choice of the full-order 
discretization. In this spirit, we assume $\hilbSPF$ to be a sufficiently well-resolving finite element space with 
$\nFe$ degrees of freedom. Similarly, we assume $\dRandSnap$ to be a large enough sample set, so that the error 
associated with the MC sampling is sufficiently small. Consequently, the errors associated with the MC sampling and the 
FE discretization are not represented in our error estimates.

\subsection{Monte Carlo reduced basis model} \label{sec:mcrbRbModel}

Let $\hilbSPR\subset\hilbSPF$ be an $\nRb$-dimensional subspace. An example is given in \cref{sec:spacesMcrb}. A reduced-order model of \cref{problem:strong} is
\begin{problem}[MCRB model]\label{problem:reducedStrong}
For given $(\rand,\para)\in\dRandSnap\times\dPara$, find
\[
  \uSPRpara\in\hilbSPR\;\colon\quad\aSpara{\uSPRpara}{\uTest} = \fSpara{\uTest}\quad\forall \uTest\in \hilbSPR.
\]
\end{problem}
The unique solvability of \cref{problem:reducedStrong} is a direct consequence of \cref{problem:strong} being well-posed and $\hilbSPR$ being a subspace of $\hilbSPF$.

\subsection{Output statistics and error estimates} \label{sec:mcrbEstimates}

\newcommand{\eP}{e}

\newcommand{\uSDF}[1]{u_{(#1)}}
\newcommand{\uSDR}[1]{\uSDF{#1}^{\nRb}}
\newcommand{\uSDFpara}[2][\rand,\para]{\uSDF{#2}(#1)}
\newcommand{\uSDRpara}[2][\rand,\para]{\uSDR{#2}(#1)}

\newcommand{\hilbSDR}[1][i]{{\hilbSPF^{\nRb}_{(#1)}}}

\newcommand{\rSP}{r}
\newcommand{\rSD}[1]{\rSP_{(#1)}}

\newcommand{\coerS}{\alpha}
\newcommand{\coerSpara}[1][\rand,\para]{\coerS(#1)}

We derive residual-corrected RB approximations of MCFE estimates of the expectation and variance of linear outputs of the parametrized PDE problem. We provide error bounds converging quadratically in terms of residual norms. In particular, we transfer the dual-based error bounds of \cite{HaasdonkEA2013}, considering the \emph{true} expectation and variance of RB outputs, to the setting of \cite{BoyavalEA2009}, considering \emph{MC approximations} of the expectation and variance. This requires an additional dual problem as well as a careful handling of different MC discretizations of the expected value, namely $\MC{\cdot}$ and $\MCM{\cdot}$ according to \eqref{eq:MonteCarlo}. Throughout this section, we assume the same dependency on the deterministic and stochastic parameters as in \cref{sec:mcrbFeModel,sec:mcrbRbModel}, but often omit an explicit notation of the parameter dependence for clarity.

We introduce a parametrized linear form $\lSPpara{\cdot}:\hilbSPF\rightarrow\RR$, assumed to be $(\rand,\para)$-uniformly bounded on $\hilbSPF$. We complement \cref{problem:strong,problem:reducedStrong} with auxiliary sets of dual problems to allow for residual-corrected output computations. For brevity, we provide the definitions and problems all at once, which results in some interconnections between the following statements:

\begin{definition}
  Subspaces $\hilbSDR[1],\dots,\hilbSDR[4]$ and linear forms $\lSD{1},\dots,\lSD{4}$ are given by
  \begin{align*}
    \hilbSDR[1]&\subset\hilbSPF,&\lSD{1}(\cdot)&:=\lSP(\cdot),\\
    \hilbSDR[2]&\subset\hilbSPF,&\lSD{2}(\cdot)&:=2(\lSP(\uSPR)-\rSP(\uSDR{1}))\lSP(\cdot),\\
    \hilbSDR[3]&\subset\hilbSPF,&\lSD{3}(\cdot)&:=\MC{\lSP(\uSPR)-\rSP(\uSDR{1})}\lSP(\cdot),\\
    \hilbSDR[4]&\subset\hilbSPF,&\lSD{4}(\cdot)&:=\MCM{\lSP(\uSPR)-\rSP(\uSDR{1})}\lSP(\cdot).
  \end{align*}
\end{definition}

\begin{problem}[dual MCFE models] \label{problem:strongDual}
For given $(\rand,\para)\in\dRandSnap\times\dPara$, find
\begin{alignat*}{4}
  \uSDF{i}\in \hilbSPF&\;\colon\quad&\aS(\uTest,\uSDF{i})&=-\lSD{i}(\uTest)&\quad&\forall \uTest \in \hilbSPF,&\quad&i=1,\dots,4.
\end{alignat*}
\end{problem}

\begin{problem}[dual MCRB models] \label{problem:strongReducedDual}
For given $(\rand,\para)\in\dRandSnap\times\dPara$, find
\begin{alignat*}{4}
  \uSDR{i}\in \hilbSDR&\;\colon\quad&\aS(\uTest,\uSDR{i})&=-\lSD{i}(\uTest)&\quad&\forall \uTest \in \hilbSDR,&\quad&i=1,\dots,4.
\end{alignat*}
\end{problem}

\begin{definition} \label{definition:strongResidual}
  A primal residual $\rSP$ and dual residuals $\rSD{1},\dots,\rSD{4}$ are  given by
  \begin{alignat}{5}
    \rSP(\cdot) &:= \fS(\cdot) - \aS(\uSPR,\cdot),\label{eq:residualStrong}\\
    \rSD{i}(\cdot) &:= -\lSD{i}(\cdot) - \aS(\cdot,\uSDR{i}),&\quad&i=1,\dots4.\label{eq:residualStrongDual}
  \end{alignat}
\end{definition}
The following error bounds require a coercivity factor
\begin{equation} \label{eq:coercivityStrong}
  \coerSpara := \inf_{\uTest\in\hilbSPF\setminus\{0\}}\frac{\aSpara{\uTest}{\uTest}}{\|\uTest\|_{\hilbSPF}^2}\quad\forall(\rand,\para)\in\dRandSnap\times\dPara,
\end{equation}
and a dual space $\hilbSDF$ of $\hilbSPF$, with norm
\begin{equation} \label{eq:dualNormStrong}
  \|F\|_{\hilbSDF} := \sup_{v\in \hilbSPF\setminus\{0\}}\frac{|F(v)|}{\|v\|_{\hilbSPF}}\quad\forall F\in \hilbSDF.
\end{equation}
For efficiency, an offline/online decomposition of the dual norms of the encountered functional is possible and the coercivity factor can be replaced by a strictly positive lower bound \cite{BoyavalEA2009,HaasdonkEA2013}.

First, we provide a bound for the error of the RB solution to \cref{problem:reducedStrong} with respect to the FE 
solution to \cref{problem:strong}, point-wise in $\dRandSnap\times\dPara$, see \citep[Proposition 3.1]{HaasdonkEA2013}:
\begin{lemma}[solution bound]\label{theorem:errorReducedStrong}
  For given $(\rand,\para)\in\dRandSnap\times\dPara$,
  \begin{align} \label{eq:errorReducedStrong}
    \|\uSPF-\uSPR\|_\hilbSPF \leq \frac{\|\rSP\|_{\hilbSDF}}{\coerS}.
  \end{align}
\end{lemma}
\begin{proof} We define $\eP := \uSPF-\uSPR$ and derive
  \[
    \coerS\|\eP\|_{\hilbSPF}^2 \stackrel{\eqref{eq:coercivityStrong}}{\leq} \aS(\eP,\eP) \stackrel{\eqref{eq:strongPrimal}}{=} \fS(\eP) - \aS(\uSPR,\eP) \stackrel{\eqref{eq:residualStrong}}{=} \rSP(\eP) \stackrel{\eqref{eq:dualNormStrong}}{\leq} \|\rSP\|_{\hilbSDF}\|\eP\|_{\hilbSPF}.
  \]
  Dividing by $\coerS\|\eP\|_{\hilbSPF}$ gives the result.
\end{proof}

We approximate a parameter-dependent linear output $\lSPpara{\uSPFpara}$ point-wise with a residual-corrected 
reduced-order approximation, see \citep[Theorem 3.6]{HaasdonkEA2013}:
\begin{lemma}[output bound] \label{theorem:functionalBoundsDualBased}
  For given $(\rand,\para)\in\dRandSnap\times\dPara$,
  \begin{align} \label{eq:functionalBoundsDualBased}
    |\lSP(\uSPF)-\lSP(\uSPR) + \rSP(\uSDR{1})| &\leq\frac{\|\rSP\|_{\hilbSDF} \|\rSD{1}\|_{\hilbSDF}}{\coerS}.
  \end{align}
\end{lemma}
\begin{proof} We define $\eP := \uSPF-\uSPR$ and reformulate
  \begin{align*}
    &|\lSP(\uSPF)-\lSP(\uSPR) + \rSP(\uSDR{1})| \stackrel{\eqref{eq:residualStrong}}{=} |\lSP(\eP) + \fS(\uSDR{1}) - \aS(\uSPR,\uSDR{1})|
    \stackrel{\eqref{eq:strongPrimal}}{=} |\lSP(\eP)+\aS(\eP,\uSDR{1})|\\
    &\qquad \stackrel{\eqref{eq:residualStrongDual}}{=} |\rSD{1}(\eP)| \stackrel{\eqref{eq:dualNormStrong}}{\leq} \|\rSD{1}\|_{\hilbSDF}\|\eP\|_{\hilbSPF}
    \stackrel{\eqref{eq:errorReducedStrong}}{\leq} \frac{\|\rSD{1}\|_{\hilbSDF}\|\rSP\|_{\hilbSDF}}{\coerS}.
  \end{align*}
\end{proof}

We approximate the MCFE estimate $\MC{\lSPpara[\cdot,\para]{\uSPFpara[\cdot,\para]}}$ of the parameter-dependent 
expected linear output as follows, see \citep[Corollary 4.2.]{HaasdonkEA2013}:
\begin{theorem}[expected output bound] \label{theorem:expectationBoundsStrong} For given $\para\in\dPara$,
  \begin{align*}
    &|\MC{\lSP(\uSPF)}-\MC{\lSP(\uSPR)} + \MC{\rSP(\uSDR{1})}| \leq\MC[\Bigg]{\frac{\|\rSP\|_{\hilbSDF} \|\rSD{1}\|_{\hilbSDF}}{\coerS}}.
  \end{align*}
\end{theorem}
\begin{proof}
  By Jensen's inequality
  \[
    |\MC{\lSP(\uSPF)}-\MC{\lSP(\uSPR)}+\MC{\rSP(\uSDR{1})}| 
    \leq \MC{|\lSP(\uSPF)-\lSP(\uSPR)+\rSP(\uSDR{1})|} 
    \stackrel{\eqref{eq:functionalBoundsDualBased}}{\leq}\MC[\Bigg]{\frac{\|\rSP\|_{\hilbSDF} \|\rSD{1}\|_{\hilbSDF}}{\coerS}}.
  \]
\end{proof}

Finally, we approximate the MCFE estimate $\MCV{\lSPpara[\cdot,\para]{\uSPFpara[\cdot,\para]}}$ of the 
parameter-dependent variance of the linear output, see \citep[Theorem 4.5]{HaasdonkEA2013}:
\begin{theorem}[output variance bound]\label{theorem:varianceBoundsDualBased} For given $\para\in\dPara$,
  \begin{align*}
    & \big| 
        \MCV{\lSP(\uSPF)}
      - \MCV{\lSP(\uSPR)}
      + \MCV{\rSP(\uSDR{1})}
      + \MCM{\rSP(\uSDR{2})}
      - \MCM{\rSP(\uSDR{3})}
      - \MC{\rSP(\uSDR{4})}
      \big|\\
    &\qquad
    \leq 
      \MCM[\Bigg]{
        \frac{
          \|\rSP\|_{\hilbSDF}^2 \|\rSD{1}\|_{\hilbSDF}^2
        }{
          \coerS^2
        }
      }
    + \MC[\Bigg]{
        \frac{
          \|\rSP\|_{\hilbSDF} \|\rSD{1}\|_{\hilbSDF}
        }{
          \coerS
        }
      }
      \MCM[\Bigg]{
        \frac{
          \|\rSP\|_{\hilbSDF} \|\rSD{1}\|_{\hilbSDF}
        }{
          \coerS
        }
      }\\
    &\qquad\qquad
    + \MCM[\Bigg]{
        \frac{
          \big\|
            \rSD{2} 
          - \rSD{3} 
          - \frac{\nRandSnap-1}{\nRandSnap}\rSD{4}
          \big\|_{\hilbSDF}
          \big\|
            \rSP
          \big\|_{\hilbSDF}
        }{
          \coerS
        }
      }.
  \end{align*}
\end{theorem}
\begin{proof}
  By \eqref{eq:MonteCarlo}, defining $\eP := \uSPF-\uSPR$,
  \begin{align*}
    & \big|
        \MCM{\lSP(\uSPF)^2}
      - \MCM{\lSP(\uSPF)}\MC{\lSP(\uSPF)}
      - \MCM{\lSP(\uSPR)^2}
      + \MCM{\lSP(\uSPR)}\MC{\lSP(\uSPR)} \\&\qquad
      + \MCM{\rSP(\uSDR{1})^2}
      - \MCM{\rSP(\uSDR{1})}\MC{\rSP(\uSDR{1})}
      + \MCM{\rSP(\uSDR{2})}
      - \MCM{\rSP(\uSDR{3})}
      - \MC{\rSP(\uSDR{4})}
      \big|\\ &\qquad\qquad
    \leq
      \underbrace{
        \Big|
          \MCM{(\lSP(\uSPF)-\lSP(\uSPR)+\rSP(\uSDR{1}))^2}
        \Big|
      }_{
        \stackrel{\eqref{eq:functionalBoundsDualBased}}\leq \MCM[\Big]{\frac{\|\rSP\|_{\hilbSDF}^2 \|\rSD{1}\|_{\hilbSDF}^2}{\coerS^2}}
      } \\&\qquad\qquad\qquad
    + \underbrace{
        \Big|
          \MCM{\lSP(\uSPF)-\lSP(\uSPR)+\rSP(\uSDR{1})}
        \Big|
      }_{
        \stackrel{\eqref{eq:functionalBoundsDualBased}}\leq \MCM[\Big]{\frac{\|\rSP\|_{\hilbSDF} \|\rSD{1}\|_{\hilbSDF}}{\coerS}}
      }
      \underbrace{
        \Big|
          \MC{\lSP(\uSPF)-\lSP(\uSPR)+\rSP(\uSDR{1})}
        \Big|
      }_{
        \stackrel{\eqref{eq:functionalBoundsDualBased}}\leq \MC[\Big]{\frac{\|\rSP\|_{\hilbSDF} \|\rSD{1}\|_{\hilbSDF}}{\coerS}}
      } \\&\qquad\qquad\qquad
    + \Big|\;
        \MCM{\;
          \underbrace{
            2(\lSP(\uSPR) - \rSP(\uSDR{1}))\lSP(\eP) + \rSP(\uSDR{2})
          }_{
            \stackrel{\eqref{eq:residualStrong},\eqref{eq:strongPrimal},\eqref{eq:residualStrongDual}}= -\rSD{2}(\eP)
          }\;
        }
      - \MCM{\;
          \underbrace{
            \MC{\lSP(\uSPR) - \rSP(\uSDR{1})}\lSP(\eP) + \rSP(\uSDR{3})
          }_{
            \stackrel{\eqref{eq:residualStrong},\eqref{eq:strongPrimal},\eqref{eq:residualStrongDual}}= -\rSD{3}(\eP)
          }\;
        } \\&\qquad\qquad\qquad\qquad
      - \MC{\;
          \underbrace{
            \MCM{\lSP(\uSPR) - \rSP(\uSDR{1})}\lSP(\eP) + \rSP(\uSDR{4})
          }_{
            \stackrel{\eqref{eq:residualStrong},\eqref{eq:strongPrimal},\eqref{eq:residualStrongDual}}= -\rSD{4}(\eP)
          }\;
        }\;
      \Big|,
  \end{align*}
  where 
  \[
    \big|
        \MCM{\rSD{2}(\eP)}
      - \MCM{\rSD{3}(\eP)}
      - \MC {\rSD{4}(\eP)}
      \big|
    \stackrel{\eqref{eq:dualNormStrong},\eqref{eq:errorReducedStrong}}\leq
      \MCM[\Bigg]{
        \frac{
        \big\|
          \rSD{2} 
        - \rSD{3}
        - \frac{\nRandSnap-1}{\nRandSnap}\rSD{4}
        \big\|_{\hilbSDF}
        \big\|
          \rSP
        \big\|_{\hilbSDF}
        }{\coerS}
      }.
  \]
\end{proof}

\section{Stochastic Galerkin reduced basis method} \label{sec:sgrb}

\newcommand{\hilbWPF}{\bar\hilbSPF}

In the following, we replace the Monte Carlo sampling by a stochastic Galerkin procedure. We provide a full-order model (\cref{sec:sgrbFeModel}) and a reduced-order model (\cref{sec:sgrbRbModel}) to approximate the stochastic solution of the PDE problem for a given realization of the deterministic input parameters. The computation of statistics of linear outputs are described in \cref{sec:sgrbRbEstimates}, together with the respective error bounds. The computation can be separated into an expensive offline phase and a inexpensive online phase by standard means \cite{PrudhommeEA2002}.

\subsection{Stochastic Galerkin finite element model} \label{sec:sgrbFeModel}

\newcommand{\nSg}{M_\text{SG}}

\newcommand{\hilbWDF}{\bar\hilbSPF'}
\newcommand{\hilbWPR}[1][\nRb]{\hilbWPF^#1}

\newcommand{\aW}{\bar\aS}
\newcommand{\aWpara}[3][\para]{\aW(#2,#3;#1)}

\newcommand{\fW}{\bar\fS}
\newcommand{\fWpara}[2][\para]{\fW(#2;#1)}

\newcommand{\lW}{\bar\lSP}
\newcommand{\lWpara}[2][\para]{\lW(#2;#1)}

We introduce a stochastic Galerkin discretization space $\hilbSg\subset\lebesgue{\randomVariable}{\dRand}$. An example is given in \cref{sec:experimentsDiscretization}. We define the product space $\hilbWPF := \hilbSg\otimes\hilbSPF$, which is a Hilbert space with inner product $(\cdot,\cdot)_{\hilbWPF} := (\cdot,\cdot)_{\bochner[\randomVariable]{\dRand}{\hilbSPF}}$ and norm $\|\cdot\|_{\hilbWPF} := \|\cdot\|_{\bochner[\randomVariable]{\dRand}{\hilbSPF}}$ in terms of the Bochner-type space $\banach[2]{\randomVariable}{\dRand}{\hilbSPF}:=\{\,v:\dRand\rightarrow\hilbSPF\mid\int_{\dRand} \|v(\rand)\|_{\hilbSPF}^2\density[\randomVariable](\rand) \dd\rand<\infty\,\}$.

We derive a spatial-stochastic weak formulation by taking the expectation of \eqref{eq:strongPrimal}. Defining $\aWpara{\uTrial}{\uTest} := \E{\aSpara[\cdot,\para]{\uTrial}{\uTest}}$ and $\fWpara{\uTest} := \E{\fSpara[\cdot,\para]{\uTest}}$ provides
\begin{problem}[SGFE model]\label{problem:discretizedWeak}
  For given $\para\in\dPara$, find
\begin{equation} \label{eq:discretizedWeak}
  \uWPFpara\in \hilbWPF\;\colon\quad \aWpara{\uWPFpara}{\uTest} = \fWpara{\uTest}\quad\forall \uTest\in \hilbWPF.
\end{equation}
\end{problem}
As a consequence of the coercivity and boundedness properties associated with \cref{problem:strong}, the bilinear form $\aWpara{\cdot}{\cdot}$ is $\para$-uniformly bounded and coercive on $\hilbWPF$ and the linear form $\fWpara{\cdot}$ is $\para$-uniformly bounded on $\hilbWPF$. Therefore, \cref{problem:discretizedWeak} has a unique solution for any given $\para\in\dPara$ according to the Lax-Milgram lemma.

\subsection{Stochastic Galerkin reduced basis model} \label{sec:sgrbRbModel}

We introduce an $\nRb$-dimensional reduced space $\hilbWPR\subset\hilbWPF$. Suitable reduced spaces are provided in \cref{sec:spacesSgrb}. A reduced form of \cref{problem:discretizedWeak} is given as follows:
\begin{problem}[SGRB model] \label{problem:reducedWeak}
  For given $\para\in\dPara$, find
\[
  \uWPRpara\in\hilbWPR\;\colon\quad \aWpara{\uWPRpara}{\uTest} = \fWpara{\uTest}\quad\forall \uTest\in \hilbWPR.
\]
\end{problem}
The subspace property $\hilbWPR\subset\hilbWPF$ and the well-posedness of \cref{problem:discretizedWeak} imply that \cref{problem:reducedWeak} has a unique solution.

\subsection{Output statistics and error estimates} \label{sec:sgrbRbEstimates}

\newcommand{\hilbWDR}[1]{\hilbWPR_{(#1)}}

\newcommand{\uWDF}[1]{\uWPF_{(#1)}}
\newcommand{\uWDFpara}[2][\para]{\uWDF{#2}(#1)}
\newcommand{\uWDR}[1]{\uWPR_{(#1)}}
\newcommand{\uWDRpara}[2][\para]{\uWDR{#2}(#1)}

\newcommand{\coerW}{\bar\coerS}
\newcommand{\coerWpara}[1][\para]{\coerW(#1)}

\newcommand{\contWD}[1]{\bar\gamma_{(#1)}}
\newcommand{\contWDpara}[2][\para]{\contWD{#2}(#1)}

\newcommand{\rWP}{\bar\rSP}
\newcommand{\rWD}[1]{\rWP_{(#1)}}

We derive SGRB approximations of the expectation and variance of linear outputs together with error bounds with respect to the corresponding SGFE approximations. The variance can be interpreted in terms of quadratic outputs. We follow the ideas of \cite{HuynhEA2006,Sen2007} to derive the respective error bounds.

We introduce a linear form $\lWpara{\uTest} := \E{\lSPpara[\cdot,\para]{\uTest}}$, which is $\para$-uniformly bounded on $\hilbWPF$. We complement the primal problem of \cref{sec:sgrbFeModel} with corresponding dual problems:
\begin{problem}[dual SGFE models]\label{problem:weakDual}
For given $\para\in\dPara$, find 
\begin{alignat*}{3}
  \uWDF{1}\in\hilbWPF &\;\colon\quad& \aW(\uTest,\uWDF{1}) &= -\lW(\uTest) &\quad& \forall\uTest\in\hilbWPF,\\
  \uWDF{2}\in\hilbWPF &\;\colon\quad& \aW(\uTest,\uWDF{2}) &= -\E{\lSP(\uWPF+\uWPR)\lSP(\uTest)} &\quad& \forall\uTest\in\hilbWPF,\\
  \uWDF{3}\in\hilbWPF &\;\colon\quad& \aW(\uTest,\uWDF{3}) &= -\big(\lW(\uWPF)+\lW(\uWPR)-2\rWP(\uWDR{1})\big)\lW(\uTest) &\quad& \forall\uTest\in\hilbWPF.
\end{alignat*}
\end{problem}
Letting $\hilbWDR{1}\subset\hilbWPF$, $\hilbWDR{2}\subset\hilbWPF$ and $\hilbWDR{3}\subset\hilbWPF$ be $\nRb$-dimensional subspaces, we introduce the following set of reduced dual equations:
\begin{problem}[dual SGRB models]\label{problem:weakReducedDual}
For given $\para\in\dPara$, find
\begin{alignat*}{3}
  \uWDR{1}\in\hilbWDR{1} &\;\colon\quad& \aW(\uTest,\uWDR{1}) &= -\lW(\uTest) &\quad& \forall\uTest\in\hilbWDR{1},\\
  \uWDR{2}\in\hilbWDR{2} &\;\colon\quad& \aW(\uTest,\uWDR{2}) &= -2\E{\lSP(\uWPR)\lSP(\uTest)} &\quad& \forall\uTest\in\hilbWDR{2},\\
  \uWDR{3}\in\hilbWDR{3} &\;\colon\quad& \aW(\uTest,\uWDR{3}) &= -2\big(\lW(\uWPR)-\rWP(\uWDR{1})\big)\lW(\uTest) &\quad& \forall\uTest\in\hilbWDR{3}.
\end{alignat*}
\end{problem}
The error bounds will be provided in terms of dual norms of residuals:
\begin{definition}\label{definition:residualWeak}
  Based on \cref{problem:discretizedWeak,problem:reducedWeak,problem:weakDual,problem:weakReducedDual},
  \begin{align}
    \rWP   (\cdot) &:=  \fW(\cdot) - \aW(\uWPR,\cdot), \label{eq:weakPrimalResidual}\\
    \rWD{1}(\cdot) &:= -\lW(\cdot) - \aW(\cdot,\uWDR{1}), \label{eq:weakDualResidual1}\\
    \rWD{2}(\cdot) &:= -2\E{\lSP(\uWPR)\lSP(\cdot)} - \aW(\cdot,\uWDR{2}), \label{eq:weakDualResidual2}\\
    \rWD{3}(\cdot) &:= -2\big(\lW(\uWPR)-\rWP(\uWDR{1})\big)\lW(\cdot) - \aW(\cdot,\uWDR{3}) \label{eq:weakDualResidual3}.
  \end{align}
\end{definition}
We define, for any $\para\in\dPara$, the coercivity factor
\[
  \coerWpara = \inf_{\uTest\in\hilbWPF\setminus\{0\}}\frac{\aWpara{\uTest}{\uTest}}{\|\uTest\|_{\hilbWPF}^2}
\]
and the continuity factor
\begin{equation} \label{eq:continuityVarianceWeak}
  \contWDpara{2} = \sup_{\uTrial,\uTest\in\hilbWPF\setminus\{0\}}\frac{\E{\lSPpara[\para,\cdot]{\uTrial}\lSPpara[\para,\cdot]{\uTest}}}{\|\uTrial\|_{\hilbWPF}\|\uTest\|_{\hilbWPF}}.
\end{equation}
It is possible to replace these factors by efficiently computable upper and lower bounds \cite{HuynhEA2006}. We introduce the dual space $\hilbWDF$ of $\hilbWPF$ with norm
\begin{equation} \label{eq:dualNormWeak}
  \|F\|_{\hilbWDF} := \sup_{v\in \hilbWPF\setminus\{0\}}\frac{|F(v)|}{\|v\|_{\hilbWPF}} \qquad \forall F\in\hilbWDF.
\end{equation}

We can derive the following error bound for the error in the reduced-order approximation of the solution:
\begin{lemma}[solution bound]\label{theorem:errorReducedWeak}
  For given $\para\in\dPara$,
  \begin{equation} \label{eq:errorReducedWeak}
    \|\uWPF-\uWPR\|_{\hilbWPF} \leq \frac{\|\rWP\|_{\hilbWDF}}{\coerW}.
  \end{equation}
\end{lemma}
\begin{proof} Analog to the proof of \cref{theorem:errorReducedStrong}.
\end{proof}

In view of the definition of the weak linear form $\lW$, we obtain the following bounds for the expected value and variance of the output:
\begin{theorem}[expected output bound] \label{theorem:expectationBoundsWeak}
For given $\para\in\dPara$,
\begin{equation} \label{eq:expectationBoundsWeak}
  |\E{\lSP(\uWPF)}-\lW(\uWPR)+\rWP(\uWDR{1})| \leq \frac{\|\rWP\|_{\hilbWDF}\|\rWD{1}\|_{\hilbWDF}}{\coerW}.
\end{equation}
\end{theorem}
\begin{proof}
  Analog to the proof of \cref{theorem:functionalBoundsDualBased}.
\end{proof}
\begin{theorem}[output variance bound] \label{theorem:varianceBoundsWeak}
For given $\para\in\dPara$,
\begin{align}
  & \big|
      \V{\lSP(\uWPF)}
    - \V{\lSP(\uWPR)}
    - \rWP(\uWDR{1})^2
    + \rWP(\uWDR{2})
    - \rWP(\uWDR{3})
    \big|\notag\\
  &\qquad 
  \leq
    \frac{
      \contWD{2}
      \|
        \rWP
      \|_{\hilbWDF}^2
    }{\coerW^2}
  +
    \frac{
      \|
        \rWP
      \|_{\hilbWDF}^2
      \|
        \rWD{1}
      \|_{\hilbWDF}^2
    }{\coerW^2}
  +
    \frac{
      \|
        \rWD{2}
      - \rWD{3}
      \|_{\hilbWDF}
      \|
        \rWP
      \|_{\hilbWDF}
    }{\coerW}.\label{eq:varianceBoundsWeak}
\end{align}
\end{theorem}
\begin{proof}
  Setting $\eP = \uWPF - \uWPR$, we rewrite
  \begin{align}
        \rWP(\uWDR{2}) 
      &\stackrel{\eqref{eq:discretizedWeak},\eqref{eq:weakPrimalResidual}}= 
        \aW(\eP,\uWDR{2}) 
      \stackrel{\eqref{eq:weakDualResidual2}}= 
      - 2\E{\lSP(\uWPR)\lSP(\eP)} - \rWD{2}(\eP), \label{eq:varianceBoundsWeakProof2}
    \\
        \rWP(\uWDR{3}) 
      &\stackrel{\eqref{eq:discretizedWeak},\eqref{eq:weakPrimalResidual}}=
        \aW(\eP,\uWDR{3}) 
      \stackrel{\eqref{eq:weakDualResidual3}}= 
      - 2(\lW(\uWPR)-\rWP(\uWDR{1}))\lW(\eP) - \rWD{3}(\eP). \label{eq:varianceBoundsWeakProof3}
  \end{align}
  After expressing the variance in terms of expectations, the left-hand side of \eqref{eq:varianceBoundsWeak} becomes
  \begin{align*}
  & \big|
      \underbrace{
        \E{\lSP(\uWPF)^2}
      - \E{\lSP(\uWPR)^2}
      + \rWP(\uWDR{2})
      }_{
      = \E{\lSP(\eP)^2} - \rWD{2}(\eP) \;\text{by}\;\eqref{eq:varianceBoundsWeakProof2}
      } 
      \underbrace{
      - \lW(\uWPF)^2
      + \lW(\uWPR)^2
      - \rWP(\uWDR{1})^2
      - \rWP(\uWDR{3})
      }_{
      = \rWD{3}(\eP) - (\lW(\eP) + \rWP(\uWDR{1}))^2 \;\text{by}\;\eqref{eq:varianceBoundsWeakProof3}
      }
    \big|
  \\ &\qquad
  \leq
    \big|
      \E{\lSP(\eP)^2}
    \big|
  + 
    (\lW(\eP) + \rWP(\uWDR{1}))^2
  + \big|
      \rWD{2}(\eP) - \rWD{3}(\eP)
    \big|.
\end{align*}
The final result follows from the following bounds:
\begin{align}
      \big|
        \E{\lSP(\eP)^2}
      \big|
    &\stackrel{\eqref{eq:continuityVarianceWeak}}\leq
      \contWD{2}
      \|
        \eP
      \|_{\hilbWPF}^2
    \stackrel{\eqref{eq:errorReducedWeak}}\leq
      \frac{ 
        \contWD{2}
        \|
          \rWP
        \|_{\hilbWDF}^2
      }{\coerW^2}, \label{eq:varianceBoundsWeakProofContinuity}
  \\
      (
        \lW(\eP) 
      + \rWP(\uWDR{1})
      )^2 
    &\stackrel{\eqref{eq:expectationBoundsWeak}}\leq
      \frac{
        \|
          \rWP
        \|_{\hilbWDF}^2
        \|
          \rWD{1}
        \|_{\hilbWDF}^2
      }{\coerW^2},\label{eq:varianceBoundsWeakProofSquaredResidual}
  \\
    \big|
      \rWD{2}(\eP) 
    - \rWD{3}(\eP)
    \big| 
    &\stackrel{\eqref{eq:dualNormWeak}}\leq
    \|
      \rWD{2}
    - \rWD{3}
    \|_{\hilbWDF}
    \|
      \eP
    \|_{\hilbWPF}
    \stackrel{\eqref{eq:errorReducedWeak}}\leq 
    \frac{
      \|
        \rWD{2}
      - \rWD{3}
      \|_{\hilbWDF}
      \|
        \rWP
      \|_{\hilbWDF}
    }{\coerW}.\notag
\end{align}
\end{proof}

\section{Reduced spaces} \label{sec:spaces}

\newcommand{\rbf}[1]{\phi_{#1}}
\newcommand{\iRb}{r}
\newcommand{\paraSample}[1]{\para^{#1}}

\newcommand{\hilb}{H}
\newcommand{\SnapshotMatrix}{\mathcal U}
\newcommand{\SnapshotVector}[1]{U_{#1}}
\newcommand{\RightWeighting}{\mathcal W}
\newcommand{\LeftWeighting}{\mathcal S}
\newcommand{\RightWeightingFactor}{\tilde{\mathcal W}}
\newcommand{\LeftWeightingFactor}{\tilde{\mathcal S}}
\newcommand{\MHilbDet}{{\mathcal M}_{\hilbSPF}}
\newcommand{\MHilbSg}{{\mathcal M}_{\hilbSg}}
\newcommand{\UP}[2]{U^{#1}_{#2}}
\newcommand{\weight}[1]{\alpha^{#1}}

\newcommand{\iPc}{q}
\newcommand{\nPc}{|\hilbSg|}

\newcommand{\nDof}{M}

We introduce candidate reduced spaces $\hilbSPR$ and $\hilbWPR$ to be used in 
\cref{problem:reducedStrong,problem:reducedWeak}, respectively.~For simplicity, we focus on spaces generated by 
snapshot-based proper orthogonal decomposition (POD), but the theory of \cref{sec:mcrb,sec:sgrb} does not dependent on 
this choice.~We point out that the availability of computable error bounds also allows the use of greedy snapshot 
sampling \cite{BoyavalEA2009,HaasdonkEA2013,NgocCuongEA2005}.

The procedures described in this section can also be applied to create the dual reduced spaces encountered in 
\cref{problem:strongReducedDual,problem:weakReducedDual}, by applying the POD to snapshots of the corresponding 
discretized dual solutions.~The creation of the dual reduced spaces must follow a certain sequence because some of the 
discretized dual problems contain reduced primal and dual solutions on their right-hand sides. For instance, creating 
$\hilbWDR{3}$ from samples of $\uWDF{3}$ requires the availability of $\hilbWPR$ and $\hilbWDR{1}$ due to the right-hand 
side of the discretized dual problem that defines $\uWDF{3}$, see \cref{problem:weakDual}.

We motivate the POD spaces by corresponding continuous minimization problems. We discretize these minimization problems 
using quadrature \citep[sections 6.4 and 6.5]{QuarteroniEA2016}. The discrete minimization problems can be solved using 
a weighted singular value decomposition of a snapshot matrix, based on \cite{KunischVolkwein1999}. \Cref{algorithm} 
provides a definition of the POD algorithm in terms of linear algebra, assuming $\nSnap$ snapshot vectors of length 
$\nDof$. The algorithm is formulated in a way that allows a general snapshot weighting and the maximum possible number 
of output vectors. Actual implementations can benefit from using a simpler (diagonal) snapshot weighting matrix and 
assuming a small number of output vectors. \Cref{sec:spacesMcrb,sec:spacesSgrb} describe how to generate the input to 
the algorithm in order to compute POD basis vectors from available FE or SGFE snapshots.

\begin{algorithm}[htbp]
\caption{Proper orthogonal decomposition.}\label{algorithm}
\begin{algorithmic}[1]
\renewcommand{\algorithmicrequire}{\textbf{Input:}}
\renewcommand{\algorithmicensure}{\textbf{Output:}}
\REQUIRE Snapshot matrix  $\SnapshotMatrix = \left(\UP{1}{},\dots,\UP{\nSnap}{}\right)\in\RR^{\nDof\times\nSnap}$. Symmetric positive definite weighting matrices $\LeftWeighting\in\RR^{\nDof\times\nDof}$ and $\RightWeighting\in\RR^{\nSnap\times\nSnap}$.
\ENSURE POD basis matrix $\Phi = (\Phi_1,\dots,\Phi_{\nDof})\in\RR^{\nDof\times\nDof}$.
\STATE Compute Cholesky factor $\LeftWeightingFactor$ such that $\LeftWeighting = \LeftWeightingFactor^T\LeftWeightingFactor$.
\STATE Compute Cholesky factor $\RightWeightingFactor$ such that $\RightWeighting = \RightWeightingFactor^T\RightWeightingFactor$.
\STATE Compute singular value decomposition $\tilde{\Phi}\Sigma\tilde{\mathcal V}^T$ of $\LeftWeightingFactor\SnapshotMatrix \RightWeightingFactor^T$.
\STATE Solve $\LeftWeightingFactor\Phi = \tilde{\Phi}$ for $\Phi$.
\end{algorithmic}
\end{algorithm}

\subsection{Spatial POD} \label{sec:spacesMcrb}

\newcommand{\podS}[1]{\varphi_{#1}}

We provide a POD of snapshots of the solution $\uSPF$ of \cref{problem:strong}, resulting in a spatial POD space $\hilbSPR=\spn(\podS{1},\dots,\podS{\nRb})$ for $\nRb\leq\nFe$. One can define a POD basis as a set of functions which solve the continuous minimization problems
\begin{equation} \label{eq:PodStrong}
  \min_{\podS{1},\dots,\podS{\nRb}\in\hilbSPF}\int_\dPara\int_{\dRand}\left\|\uSPFpara-\sum_{\iRb=1}^\nRb(\uSPFpara,\podS{\iRb})_\hilbSPF\podS{\iRb}\right\|_\hilbSPF^2\dd\probability[\randomVariable](\rand)\dd\para,\quad (\podS{i},\podS{j})_\hilbSPF=\delta_{ij},
\end{equation}
for $\nRb=1,\dots,\nFe$. We approximate the double integral using Monte Carlo quadrature. 

Concerning the Monte Carlo quadrature of the $\dRand$-integral on the one hand, we already know that the reduced-order model will only be evaluated at the random parameter points $\randSnap{1},\dots,\randSnap{\nRandSnap}$, because the Monte Carlo discretization of the final stochastic estimates is fixed from the beginning. We use exactly these points for the discretization of the POD minimization problem, too, because with this choice our reduced basis will be optimal in a mean-square sense with respect to approximating the finite element solution at $\randSnap{1},\dots,\randSnap{\nRandSnap}$.

Concerning the Monte Carlo quadrature of the $\dPara$-integral on the other hand, our model should be able to estimate 
the output statistics reasonably well at any point in $\dPara$. Having no further information about how the 
reduced-order model will ultimately be used, we discretize the deterministic parameter domain using a training set 
$\dParaSnap=\{\paraSnap{1},\dots,\paraSnap{\nParaSnap}\}$, with $\paraSnap{1},\dots,\paraSnap{\nParaSnap}$ distributed 
independently and uniformly over $\dPara$. When testing the performance of the resulting reduced-order model, we use a 
different set of points in the parameter domain in order to verify the robustness of the model with respect to the 
deterministic parameter.

The Monte Carlo quadrature of the double integral in \eqref{eq:PodStrong} finally results in a set of discretized POD minimization problems
\begin{equation} \label{eq:PodStrongQuadrature}
  \min_{\podS{1},\dots,\podS{\nRb}\in\hilbSPF}\frac1{\nRandSnap\nParaSnap}\sum_{i=1}^{\nRandSnap}\sum_{j=1}^{\nParaSnap} \left\|\uSPFpara[\randSnap{i},\paraSample{j}]-\sum_{\iRb=1}^\nRb(\uSPFpara[\randSnap{i},\paraSample{j}],\podS{\iRb})_\hilbSPF\podS{\iRb}\right\|_\hilbSPF^2,\quad (\podS{i},\podS{j})_\hilbSPF=\delta_{ij},
\end{equation}
for $\nRb=1,\dots,\nFe$. For the POD computation in terms of \Cref{algorithm}, we set $\nSnap = \nRandSnap\nParaSnap$ 
and $\nDof=\nFe$ and let $\UP{(i-1)\nRandSnap+j}{}\in\RR^{\nDof}$ be the coefficient vector corresponding to the 
expansion of $\uSPFpara[\randSnap{i},\paraSample{j}]\in\hilbSPF$ in a basis of $\hilbSPF$ for $i=1,\dots,\nRandSnap$ and 
$j=1,\dots,\nParaSnap$. By substituting the finite element basis expansions into \eqref{eq:PodStrongQuadrature}, we find
\[
  \LeftWeighting=\MHilbDet,\quad\SnapshotMatrix = \left(\UP{1}{},\dots,\UP{\nSnap}{}\right),\quad\RightWeighting=\frac1\nSnap I_{\nSnap},
\]
where $\MHilbDet$ denotes the mass matrix corresponding to $\hilbSPF$ and $I_{\nSnap}$ denotes the $\nSnap\times\nSnap$ identity matrix. The output of \Cref{algorithm} is a POD basis matrix $\Phi = (\Phi_1,\dots,\Phi_{\nFe})\in\RR^{\nFe\times\nFe}$. The $i$-th POD basis function $\podS{i}$ can be obtained from the $i$-th POD basis vector $\Phi_i$ via an expansion in the available basis of $\hilbSPF$, using the elements of $\Phi_i$ as expansion coefficients. Finally, an $\nRb$-dimensional POD reduced space is given by $\hilbSPR = \spn(\podS{1},\dots,\podS{\nRb})$ for any $\nRb=1,\dots,\nFe$ and the trivial space $\hilbSPR[0]\subset\hilbSPF$ contains only the zero function.

\subsection{Spatial-stochastic POD} \label{sec:spacesSgrb}

\newcommand{\podW}[1]{\bar\varphi_{#1}}

We introduce a reduced basis space that can be used to derive a stochastic Galerkin reduced basis method. It employs a simultaneous reduction of the spatial and stochastic degrees of freedom of a stochastic Galerkin finite element discretization.

Spatial-stochastic POD reduced basis functions can be defined as solutions to a set of $\dPara$-con\-tin\-u\-ous POD minimization problems
\begin{equation*}
  \min_{\podW{1},\dots,\podW{\nRb}\in\hilbWPF}\int_\dPara\left\|\uWPFpara-\sum_{\iRb=1}^\nRb(\uWPFpara,\podW{\iRb})_{\hilbWPF}\podW{\iRb}\right\|_{\hilbWPF}^2\dd\para,\quad (\podW{i},\podW{j})_{\hilbWPF}=\delta_{ij},
\end{equation*}
for $\nRb=1,\dots,\nFe\nSg$. A Monte Carlo quadrature of the $\dPara$-integral raises the issue of choosing the sample points. Using the same training set $\dParaSnap=\{\paraSnap{1},\dots,\paraSnap{\nParaSnap}\}$ as in \cref{sec:spacesMcrb} leads to discrete POD minimization problems
\begin{equation} \label{eq:PodWeakDiscrete}
  \min_{\podW{1},\dots,\podW{\nRb}\in\hilbWPF}\frac1\nParaSnap\sum_{\iSnap=1}^{\nParaSnap}\left\|\uWPFpara[\paraSample{\iSnap}]-\sum_{\iRb=1}^\nRb(\uWPFpara[\paraSample{\iSnap}],\podW{\iRb})_{\hilbWPF}\podW{\iRb}\right\|_{\hilbWPF}^2,\quad (\podW{i},\podW{j})_{\hilbWPF}=\delta_{ij}
\end{equation}
for $\nRb=1,\dots,\nFe\nSg$. Regarding the POD computation in terms of \Cref{algorithm}, we set $\nSnap = \nParaSnap$ and $\nDof=\nFe\nSg$ and denote the stochastic Galerkin finite element coefficient vector of $\uWPFpara[\paraSample{\iSnap}]$ by $\UP{\iSnap}{}$. By substituting the stochastic Galerkin finite element basis expansions into \eqref{eq:PodWeakDiscrete}, we obtain
\[
  \LeftWeighting = \MHilbSg\otimes\MHilbDet,\quad \SnapshotMatrix = \left(\UP{1}{},\dots,\UP{\nParaSnap}{}\right),\quad\RightWeighting=\frac1\nParaSnap I_{\nParaSnap},
\]
where $I_{\nParaSnap}$ is the $\nParaSnap\times\nParaSnap$ identity matrix and $\MHilbSg$ is the mass matrix containing 
the mutual $\hilbSg$-inner products of the basis functions used to represent $\hilbSg$. In view of \Cref{algorithm}, 
the $i$-th POD basis function $\podW{i}$ can be obtained from the $i$-th POD basis vector $\Phi_i$ via an expansion in 
the available basis of $\hilbWPF$, using the elements of $\Phi_i$ as expansion coefficients. Finally, an 
$\nRb$-dimensional POD reduced space is given by $\hilbWPR = \spn(\podW{1},\dots,\podW{\nRb})$ for any 
$\nRb=1,\dots,\nFe\nSg$ and the trivial space $\hilbWPR[0]\subset\hilbWPF$ contains only the zero function.

\section{Numerical experiments} \label{sec:experiments}

\newcommand{\spaceVar}[1][]{x_{#1}}
\newcommand{\dSpace}{\Omega}

\newcommand{\diffusion}[1][]{\kappa_{#1}}
\newcommand{\eigval}[1]{\lambda_{#1}}
\newcommand{\stdDiffusion}{\sigma}
\newcommand{\lengthDiffusion}{L}

We assess the provided error bounds and compare the accuracy of the MCRB and SGRB models in terms of computing the 
expectation and variance of a linear output for a convection-diffusion-reaction problem.

\subsection{Example problem} \label{sec:experimentsFormulation}

Let $\rand = (\rand[1],\dots,\rand[\nRand])^T\in\dRand\subset\RR^\nRand$ denote the value of a sample of a random parameter vector, $\para = (\para[1],\para[2])^T \in\dPara\subset\RR^2$ the value of a deterministic parameter vector and $\spaceVar=(\spaceVar[1],\spaceVar[2])^T\in\dSpace\subset\RR^2$ the coordinate in the computational domain $\dSpace$. We model the random input by a second-order random field with expected value $\diffusion[0]$ and separable exponential covariance $c(\spaceVar) = \stdDiffusion^2\exp(-|\spaceVar[1]|/\lengthDiffusion-|\spaceVar[2]|/\lengthDiffusion)$, where $\stdDiffusion$ is the standard deviation and $\lengthDiffusion$ is the correlation length. We approximate the random field using a truncated Karhunen-Lo\`eve expansion
$
  \diffusion(\spaceVar;\rand) = \diffusion[0] + \stdDiffusion\sum_{\iRand=1}^\nRand \sqrt{\eigval{\iRand}}\diffusion[\iRand](\spaceVar)\rand[\iRand],
$
where $\eigval{\iRand}$ denote the eigenvalues of the corresponding eigenvalue problem, ordered decreasingly by magnitude, and $\diffusion[\iRand](\spaceVar)$ denote respective eigenfunctions. The covariance function allows for an analytical solution of the eigenvalue problem \cite{GhanemSpanos1991}. We assume that the parameters of the Karhunen-Lo\`eve expansion originate from independent uniformly distributed random variables.  The truncation index $\nRand$ can be interpreted as a modeling parameter, because it enters the definition of the bilinear form. The governing equations of our example problem are provided as follows:
\begin{problem}[spatial strong form]\label{problem:test}
For any $(\rand,\para)\in\dRand\times\dPara$, find $\uSPF(\cdot;\rand,\para)\colon\dSpace\rightarrow\RR$ such that
\begin{alignat*}{2}
  \para\cdot\nabla\uSPF(\spaceVar;\rand,\para) 
- \Delta\uSPF(\spaceVar;\rand,\para) 
- \diffusion(\spaceVar;\rand)\uSPF(\spaceVar;\rand,\para)
&= 1, &\quad& \spaceVar\in\dSpace,\\
  \uSPF(\spaceVar;\rand,\para) &= 0,&&\spaceVar\in\partial\dSpace.
\end{alignat*}
\end{problem}
The deterministic parameter vector $\para\in\dPara\subset\RR^2$ can be interpreted as a spatially uniform convective velocity. The random parameter vector $\rand\in\dRand\subset\RR^\nRand$ enters via a parametrized random reactivity. A concrete instance of the example problem is determined by the model parameters given in \cref{table:modelParameters}. The output of the example problem is given by $\lSPpara{\uSPF(\rand,\para)}$, where
\begin{equation} \label{eq:functional}
  \lSPpara{\uTest} = \int_0^\frac12\int_0^\frac12\uTest(\spaceVar)\dd\spaceVar[1]\dd\spaceVar[2].
\end{equation}

\begin{table}
\caption{Model parameters of the test problem.}\label{table:modelParameters}
\begin{center}
\begin{tabular}{p{0.1\textwidth}p{0.2\textwidth}p{0.6\textwidth}}
\toprule
symbol & value & description\\
\midrule
  $\diffusion[0]$&$-1000$&expected value of reactivity\\
  $\stdDiffusion$&$200$&standard deviation of reactivity\\
  $\lengthDiffusion$&$1$&correlation length\\
  $\nRand$&$5$&Karhunen-Lo\`eve truncation index\\
  $\dSpace$&$(-0.5,0.5)^2$&spatial domain with boundary $\partial\dSpace$\\
  $\dPara$&$[-200,200]^2$&deterministic parameter domain\\
  $\dRand$&$[-\sqrt3,\sqrt3]^\nRand$&random parameter domain\\
\bottomrule
\end{tabular}
\end{center}
\end{table}

In order to express the example problem in terms of the spatial weak form of \cref{problem:strong}, we set
\begin{align}
  \aSpara{\uTrial}{\uTest}
    &= a^{0}(\uTrial,\uTest) 
    + \sum_{\iRand=1}^{\nRand}\rand[\iRand]a_{\rand}^\iRand(\uTrial,\uTest)
    + \sum_{\iPara=1}^{2}\para[\iPara]a_{\para}^\iPara(\uTrial,\uTest),\label{eq:testBilinearForm}
\end{align}
with
\begin{alignat*}{2}
    a^{0}(\uTrial,\uTest) 
  &=
    \int_\dSpace \nabla\uTrial(\spaceVar)\cdot\nabla\uTest(\spaceVar) \dd\spaceVar
  - \diffusion[0] \int_\dSpace \uTrial(\spaceVar)\uTest(\spaceVar) \dd\spaceVar, &&
  \\
    a_{\rand}^\iRand(\uTrial,\uTest)
  &=
    \stdDiffusion\sqrt{\eigval{\iRand}} \int_\dSpace \diffusion[\iRand](\spaceVar)\uTrial(\spaceVar)\uTest(\spaceVar) \dd\spaceVar, &\quad& \iRand=1,\dots,\nRand,
  \\
    a_{\para}^\iPara(\uTrial,\uTest)
  &= 
    \int_\dSpace \partial_{\spaceVar[\iPara]}\uTrial(\spaceVar)\uTest(\spaceVar) \dd\spaceVar, && \iPara=1,2
\end{alignat*}
and
\begin{equation}
  \fSpara{\uTest} = \int_\dSpace\uTest(\spaceVar)\dd\spaceVar.\label{eq:testLinearForm}
\end{equation}
A spatial weak form of \cref{problem:test} is provided in terms of the standard infinite-dimensional Sobolev space $H^1_0(\dSpace)$ as follows:
\begin{problem}[spatial weak form]\label{problem:testWeak}
For given $(\rand,\para)\in\dRand\times\dPara$, find
\begin{alignat*}{3}
  \uSPF(\rand,\para)\in H^1_0(\dSpace) &\;\colon\quad& \aSpara{\uSPF(\rand,\para)}{\uTest} &= \fSpara{\uTest} &\quad& \forall \uTest\in H^1_0(\dSpace).
\end{alignat*}
\end{problem}
By taking the expectation and using the notation of \cref{sec:sgrbFeModel}, a spatial-stochastic weak form is given by
\begin{problem}[spatial-stochastic weak form]\label{problem:testStochasticWeak}
  For given $\para\in\dPara$, find
\begin{equation*}
  \uWPFpara\in \banach[2]{\randomVariable}{\dRand}{H^1_0(\dSpace)} \;\colon\quad \aWpara{\uWPFpara}{\uTest} = \fWpara{\uTest}\quad\forall \uTest\in \banach[2]{\randomVariable}{\dRand}{H^1_0(\dSpace)}.
\end{equation*}
\end{problem}

\subsection{Discretization} \label{sec:experimentsDiscretization}

\newcommand{\nPoly}{d}

\newcommand{\sDet}{s}
\newcommand{\sDetP}[1][\rand,\para]{\sDet(#1)}

\newcommand{\sRnd}{\bar s}
\newcommand{\sRndP}[1][\para]{\sRnd(#1)}

\newcommand{\reference}[1]{#1^\text{ref}}
\newcommand{\hilbSPFRef}{\reference{\hilbSPF}}
\newcommand{\nRandSnapRef}{\reference{\nRandSnap}}
\newcommand{\nPolyRef}{\reference{\nPoly}}
\newcommand{\hilbSgRef}{\reference{\hilbSg}}
\newcommand{\nFeRef}{\reference{\nFe}}
\newcommand{\nSgRef}{\reference{\nSg}}

The MCFE and SGFE discretizations (\cref{problem:strong,problem:discretizedWeak}) provide necessary links between the infinite-dimensional test problems (\cref{problem:testWeak,problem:testStochasticWeak}) and the respective reduced-order models (\cref{problem:reducedStrong,problem:reducedWeak}). In the following, we describe the computational details of the MCFE and SGFE discretizations of the test problem. \Cref{table:discretizationParameters} lists our choice of the relevant discretization parameters.

\begin{table}
\caption{Discretization parameters of the test problem. The default values are used in the snapshot simulations. The reference values are used to assess the accuracy of the snapshot discretizations. \label{table:discretizationParameters}}
\begin{center}
\begin{tabular}{p{0.1\textwidth}p{0.1\textwidth}p{0.1\textwidth}p{0.58\textwidth}}
\toprule
symbol & default & reference & description\\
\midrule
  $\nFe$&$225$&$961$& number of FE degrees of freedom\\
  $\nRandSnap$&$1024$&$16384$&number of MC samples of $\rand\in\dRand$\\
  $\nPoly$&$2$&$3$&degree of SG polynomials\\
  $\nSg$&$243$&$1024$&number of SG degrees of freedom: $(\nPoly+1)^\nRand$\\
\bottomrule
\end{tabular}
\end{center}
\end{table}

\paragraph{Finite element method} We derive an instance of the stochastic strong finite element problem by replacing $H^1_0(\dSpace)$ in \cref{problem:testWeak} with a finite-dimensional subspace. In particular, we employ the space $\hilbSPF\subset H^1_0(\dSpace)$ formed by continuous piecewise linear finite elements corresponding to a regular graded simplicial triangulation of $\dSpace$, characterized by the number $\nFe$ of finite element degrees of freedom. We estimate the spatial discretization error using simulations on a finer reference triangulation as a substitute for the exact solution.

\paragraph{Monte Carlo method} We provide estimates of the expectation and variance by discretizing the respective stochastic integrals using Monte Carlo quadrature in the sense of \cref{sec:mcrbFeModel}. To this end, we generate random samples $\randSnap{1},\dots,\randSnap{\nRandSnap}\in\dRand$ according to the distribution $\probability[\randomVariable]$ with a standard pseudorandom number generator. A reference simulation with a higher number of samples delivers an estimate of the sampling error.

\paragraph{Stochastic Galerkin method} Stochastic Galerkin methods estimate the expectation and variance by directly evaluating the respective stochastic integrals, given a stochastic Galerkin solution based on a finite-dimensional subspace of $\lebesgue{\randomVariable}{\dRand}$. In general, a stochastic Galerkin finite element method applied to a linear elliptic problem with a random elliptic coefficient leads to a large, block-structured system of linear algebraic equations. In our case, however, the underlying random variables $\rand[1],\dots,\rand[\nRand]$ are independent and enter the bilinear form linearly, see \eqref{eq:testBilinearForm}. Under these conditions, it is possible to find a double-orthogonal polynomial basis which decouples the blocks in the system matrix \cite{BabuskaEA2004,FrauenfelderEA2005}. The resulting block-diagonal system of equations can be solved efficiently due to the relatively small bandwidth and the ability to treat the blocks in parallel. To define a suitable double-orthogonal basis, we start with $\nRand$ spaces of possibly different dimensions, spanned by univariate Legendre polynomials over the interval $[-\sqrt{3},\sqrt{3}]$. We normalize the polynomials regarding the underlying uniform distribution and rotate the bases such that they consist of double-orthogonal univariate polynomials, as described in \cite{BabuskaEA2004}. Finally, a tensor product of these univariate double-orthogonal polynomial bases forms a basis of an $\nSg$-dimensional subspace $\hilbSg\subset\lebesgue{\randomVariable}{\dRand}$. In our experiments, we use the same polynomial degree $\nPoly$ in all directions. We assess the error associated with the choice of $\nPoly$ by comparing with a reference solution using a higher degree.

\paragraph{Reduced basis} The considered reduced-order models rely on POD spaces generated from snapshots of the underlying discretized solution. We choose $\nParaSnap=64$ as the number of training samples in the deterministic parameter domain. Consequently, \cref{sec:spaces} specifies the creation of the reduced spaces.

\subsection{Results} \label{sec:experimentsReduction}

\newcommand{\sDetRed}{\sDet^\nRb}
\newcommand{\sDetRedP}[1][\rand,\para]{\sDetRed(#1)}

\newcommand{\sRndRed}{\sRnd^\nRb}
\newcommand{\sRndRedP}[1][\para]{\sRndRed(#1)}

\newcommand{\nTest}{\nSnap_{\para}^\text{test}}

\Cref{fig:parameterDependence} presents the parameter-dependent output statistics obtained with the default 
parameter-dependent SGFE model.~The crosses in \Cref{fig:parameterDependence} correspond to the snapshot training 
parameter values provided by the pseudo random number generator. Additionally, \Cref{fig:parameterDependence} shows the 
test parameter values that are used to assess the model.
\newenvironment{parameteraxis}[1][]{
  \begin{axis}[
    width = 0.27\textwidth,
    height = 0.27\textwidth,
    xlabel = {$\para[1]$},
    ylabel = {$\para[2]$},
    xtick = {-200,-100,0,100,200},
    ytick = {-200,-100,0,100,200},
    scale only axis = true,
    view = {0}{90},
    colormap access=piecewise const,
    colorbar,
    colorbar style={
      y tick label style={
        /pgf/number format/.cd,
        fixed,
        fixed zerofill,
        precision=1,
        /tikz/.cd}},
    #1]
}{
    \addplot+[only marks,mark=x,black] table {trainingParameters.dat};
    \addplot+[only marks,mark=o,black] table {testParameters.dat};
    \addplot+[only marks,mark=square,mark options={scale=1.5},black,skip coords between index={1}{64}] table {testParameters.dat};
  \end{axis}
}

\begin{figure}
  \begin{center}
    \ifUseExternalPdfs
      \includegraphics{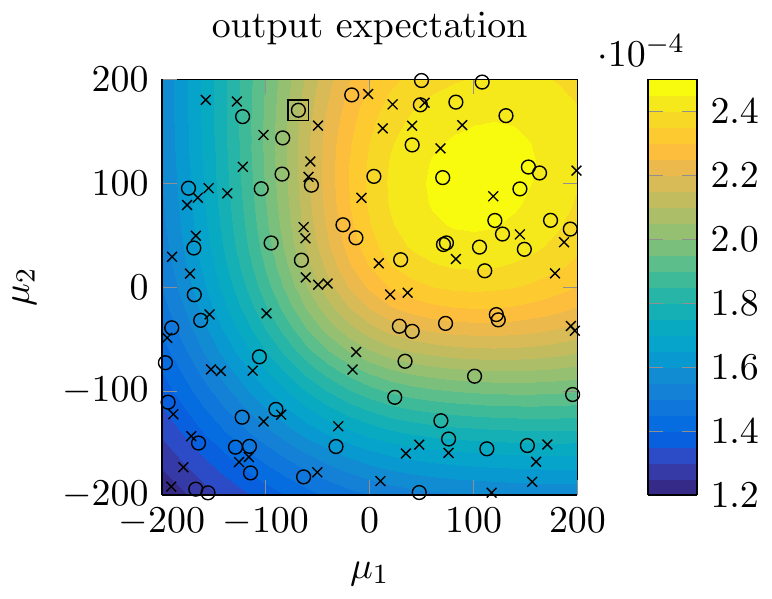}%
      \includegraphics{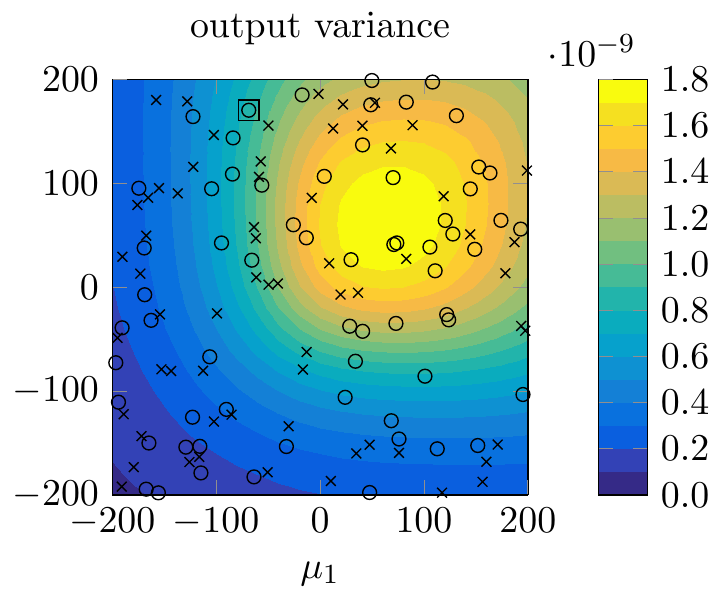}
    \else
      \begin{tikzpicture}
        \begin{parameteraxis}[
          title = {output expectation},
          point meta min=1.2e-4,
          point meta max=2.5e-4,
          colorbar style={
            ytick={
              0.00012,
              0.00014,
              0.00016,
              0.00018,
              0.00020,
              0.00022,
              0.00024}}]
          \addplot3[contour filled={number = 26},mesh/rows=21,mesh/cols=21] table {plotParameterDependenceExpectation.dat};
        \end{parameteraxis}
      \end{tikzpicture}
      \begin{tikzpicture}
        \begin{parameteraxis}[
          title = {output variance},
          point meta min=0.0e-9,
          point meta max=1.8e-9,
          colorbar style={
            ytick={
              0.0e-9,
              0.2e-9,
              0.4e-9,
              0.6e-9,
              0.8e-9,
              1.0e-9,
              1.2e-9,
              1.4e-9,
              1.6e-9,
              1.8e-9}},
          ylabel={}]
          \addplot3[contour filled={number = 18},mesh/rows=21,mesh/cols=21] table {plotParameterDependenceVariance.dat};
        \end{parameteraxis}
      \end{tikzpicture}
    \fi
  \end{center}
  \caption{Parameter-dependent output expectation $\E{\lSPpara[\cdot,\para]{\uSPF(\cdot,\para)}}$ and variance $\V{\lSPpara[\cdot,\para]{\uSPF(\cdot,\para)}}$ for the functional $\lSP$ given by \eqref{eq:functional}. Crosses mark the snapshot parameter values. The square marks the parameter value corresponding to \Cref{fig:convergencePointwise}. Circles mark the parameter values used to obtain \Cref{fig:convergenceL2}. \label{fig:parameterDependence}}
\end{figure}

The reduced basis estimates of the output statistics together with the respective error bounds are provided by 
\cref{theorem:expectationBoundsStrong,theorem:varianceBoundsDualBased} for the MCRB method and 
\cref{theorem:expectationBoundsWeak,theorem:varianceBoundsWeak} for the SGRB method.~First, we validate the error 
bounds for a single random realization of the deterministic parameter, marked by a square in 
\Cref{fig:parameterDependence}. The convergence regarding the number of reduced basis functions $\nRb$ is presented in 
\Cref{fig:convergencePointwise}. The error components of the underlying discretized solution are provided as a 
reference. Looking at the discretization errors only, we see that number of MC samples is sufficient to approximate the 
expectation but actually too small to balance the FE error in case of the variance. The SG error, on the other hand, is 
smaller than the FE error in all cases, which provides evidence that the stochastic Galerkin discretization of the 
stochastic domain is sufficiently fine. Concerning the reduced basis models, we observe that $\nRb \approx 16$ reduced 
basis functions are sufficient to obtain reduced-order estimates which are on a par with the full-order estimates in all 
considered cases. The plots suggest that all error bounds converge at the same rates as the respective errors. This is 
useful, because it implies that efficiency of the error bounds does not become significantly worse when the number of 
reduced basis functions is increased.

\newenvironment{myaxis}[1][]{
  \begin{loglogaxis}[
    width = 0.38\textwidth,
    height = 0.5\textwidth,
    xlabel= $\nRb$,
    xmin = 1,
    xmax = 64,
    xtick = {1,2,4,8,16,32,64},
    log basis x=2,
    xticklabel=\pgfmathparse{2^\tick}\pgfmathprintnumber{\pgfmathresult},
    log basis y=10,
    yticklabel=\pgfmathparse{\tick}\pgfmathprintnumber{\pgfmathresult},
    scale only axis = true,
    legend style={
      at={(0.03,0.03)},
      anchor=south west,
      cells={anchor=west},
      font=\footnotesize,
      draw=none},
    cycle list={
      {black,dashed},
      {black},
      {black,dotted},
      {black,dash dot}},
    #1]
}{
  \end{loglogaxis}
}

\begin{figure}
  \begin{center}
    \ifUseExternalPdfs
      \includegraphics{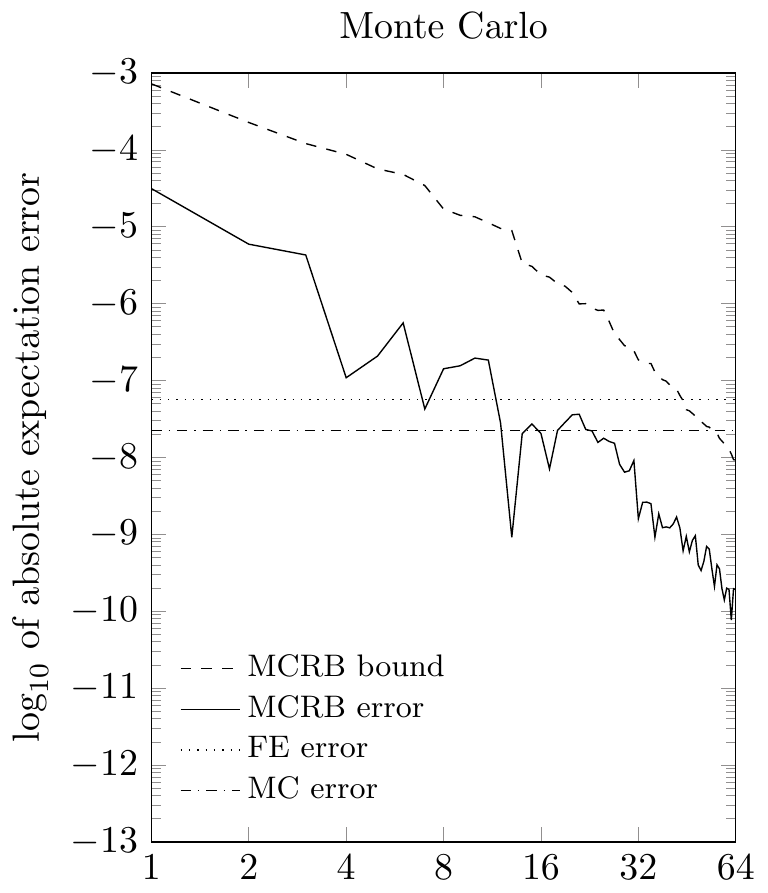}%
      \includegraphics{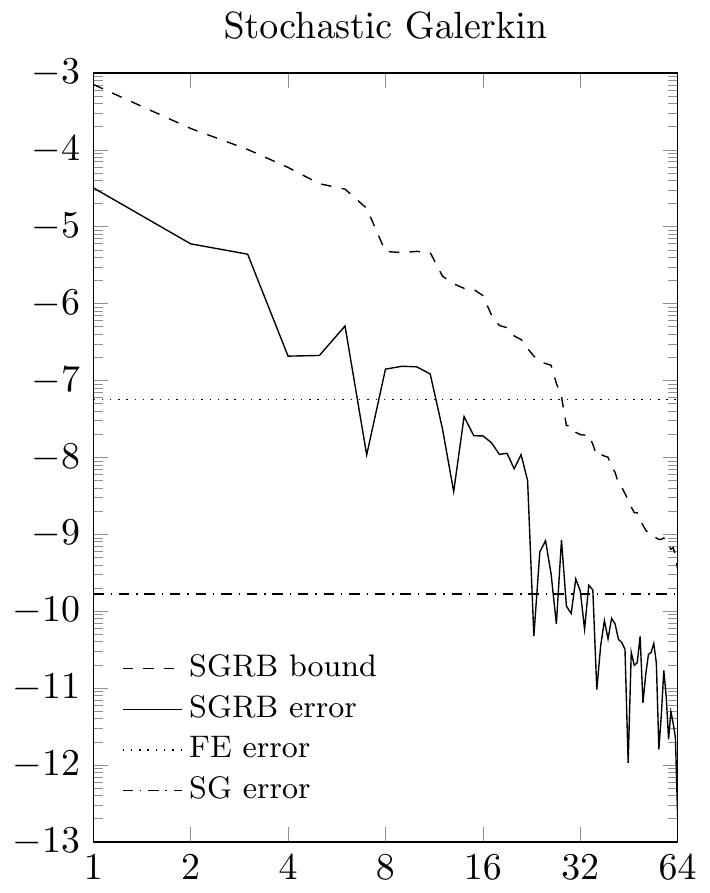}\\
      \includegraphics{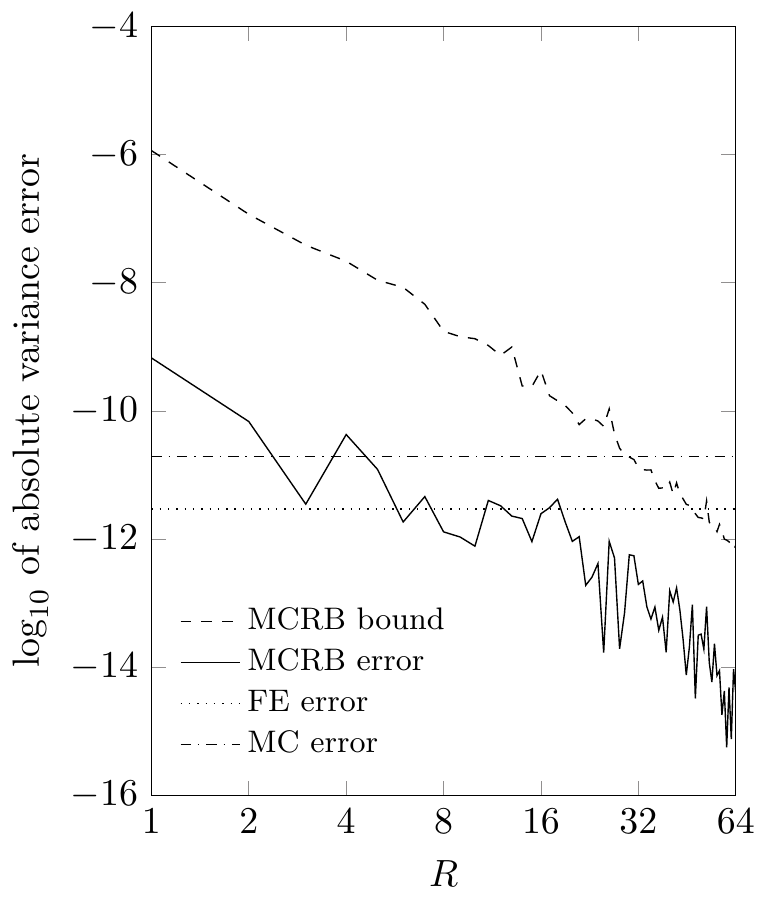}%
      \includegraphics{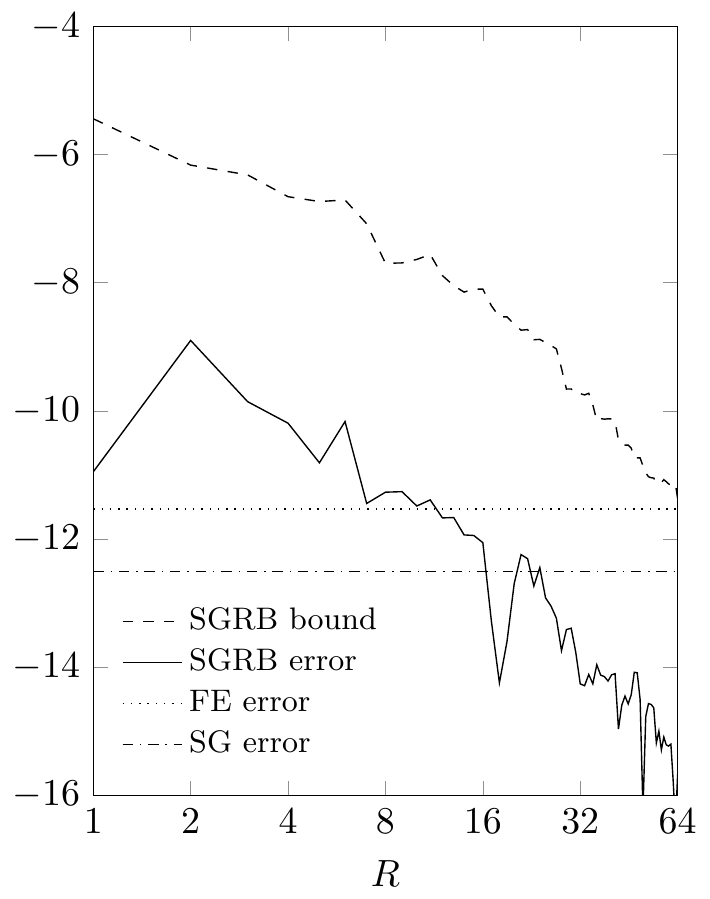}
    \else
      \begin{tikzpicture}
        \begin{myaxis}[
          ymin = 10^-13,
          ymax = 10^-3,
          xlabel = {},
          ylabel = {$\log_{10}$ of absolute expectation error},
          legend entries={
            MCRB bound,
            MCRB error,
            FE error,
            MC error},
          title = {Monte Carlo}]
          \addplot table [x=R,y=Bound]{expectationMonteCarlo.dat};
          \addplot table [x=R,y=Error]{expectationMonteCarlo.dat};
          \addplot table [x=R,y=Error]{expectationFiniteElementReference.dat};
          \addplot table [x=R,y=Error]{expectationMonteCarloReference.dat};
        \end{myaxis}
      \end{tikzpicture}
      \begin{tikzpicture}
        \begin{myaxis}[
          ymin = 10^-13,
          ymax = 10^-3,
          xlabel = {},
          legend entries={
            SGRB bound,
            SGRB error,
            FE error,
            SG error},
          title = {Stochastic Galerkin}]
          \addplot table [x=R,y=Bound]{expectationStochasticGalerkin.dat};
          \addplot table [x=R,y=Error]{expectationStochasticGalerkin.dat};
          \addplot table [x=R,y=Error]{expectationFiniteElementReference.dat};
          \addplot table [x=R,y=Error]{expectationStochasticGalerkinReference.dat};
        \end{myaxis}
      \end{tikzpicture}\\
      \begin{tikzpicture}
        \begin{myaxis}[
          ymin = 10^-16,
          ymax = 10^-4,
          ylabel = {$\log_{10}$ of absolute variance error},
          legend entries={
            MCRB bound,
            MCRB error,
            FE error,
            MC error}]
          \addplot table [x=R,y=Bound]{varianceMonteCarlo.dat};
          \addplot table [x=R,y=Error]{varianceMonteCarlo.dat};
          \addplot table [x=R,y=Error]{varianceFiniteElementReference.dat};
          \addplot table [x=R,y=Error]{varianceMonteCarloReference.dat};
        \end{myaxis}
      \end{tikzpicture}
      \begin{tikzpicture}
        \begin{myaxis}[
          ymin = 10^-16,
          ymax = 10^-4,
          legend entries={
            SGRB bound,
            SGRB error,
            FE error,
            SG error}]
          \addplot table [x=R,y=Bound]{varianceStochasticGalerkin.dat};
          \addplot table [x=R,y=Error]{varianceStochasticGalerkin.dat};
          \addplot table [x=R,y=Error]{varianceFiniteElementReference.dat};
          \addplot table [x=R,y=Error]{varianceStochasticGalerkinReference.dat};
        \end{myaxis}
      \end{tikzpicture}
    \fi
  \end{center}
  \caption{Log-log plots of the errors in the approximation of the expectation (top row) and variance (bottom row) of a linear functional with an MCRB method (left column) and an SGRB method (right column) depending on the dimension $\nRb$ of the reduced spaces, for a random point in the deterministic parameter domain. Respective error bounds and approximate FE/SG/MC discretization errors.\label{fig:convergencePointwise}}
\end{figure}

\begin{figure}
  \begin{center}
     \ifUseExternalPdfs
      \includegraphics{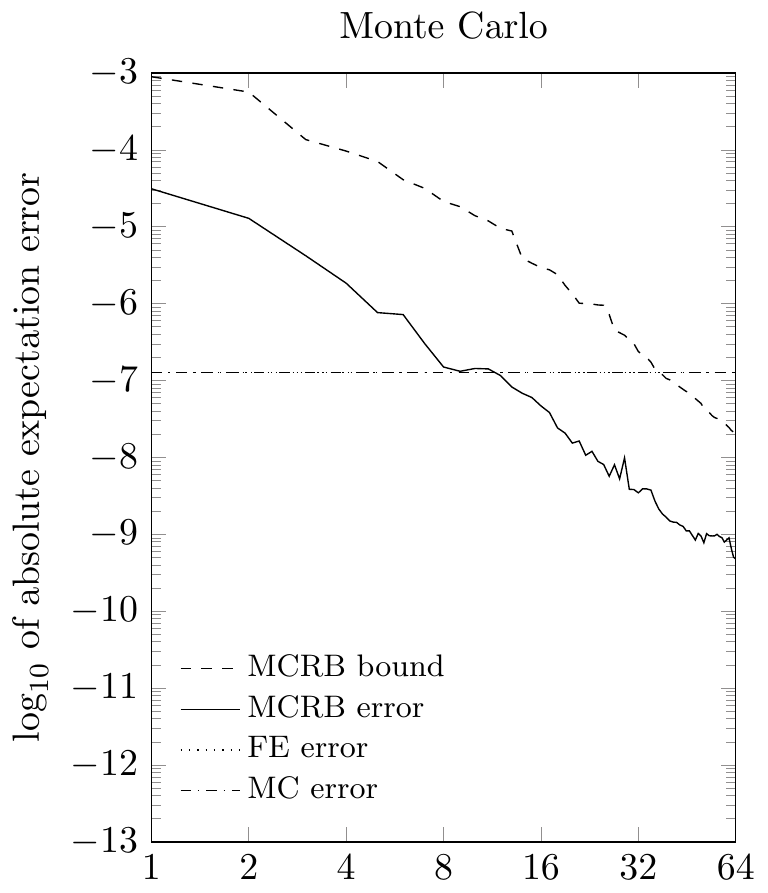}%
      \includegraphics{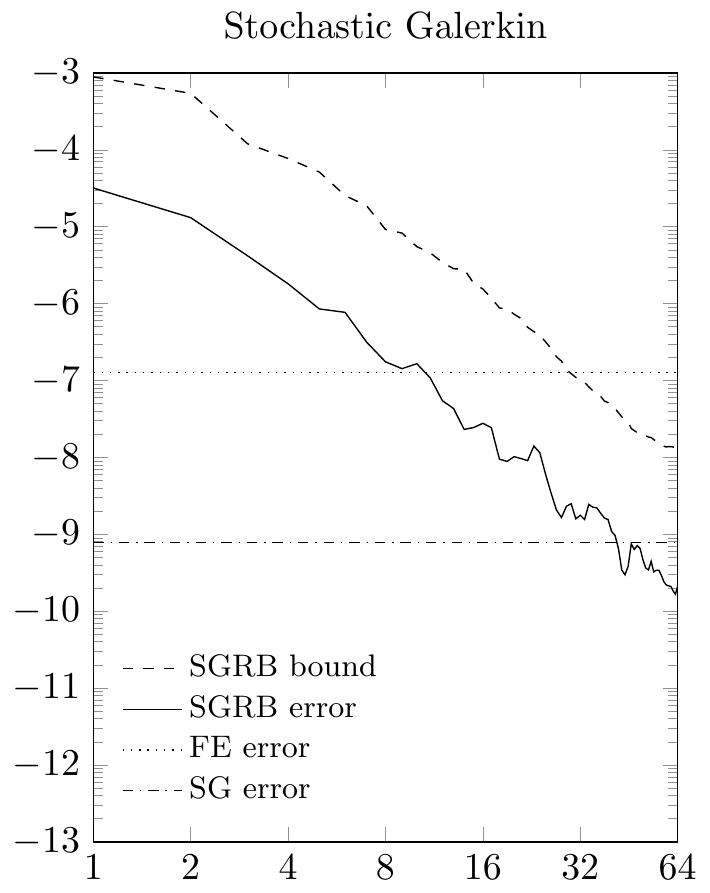}\\
      \includegraphics{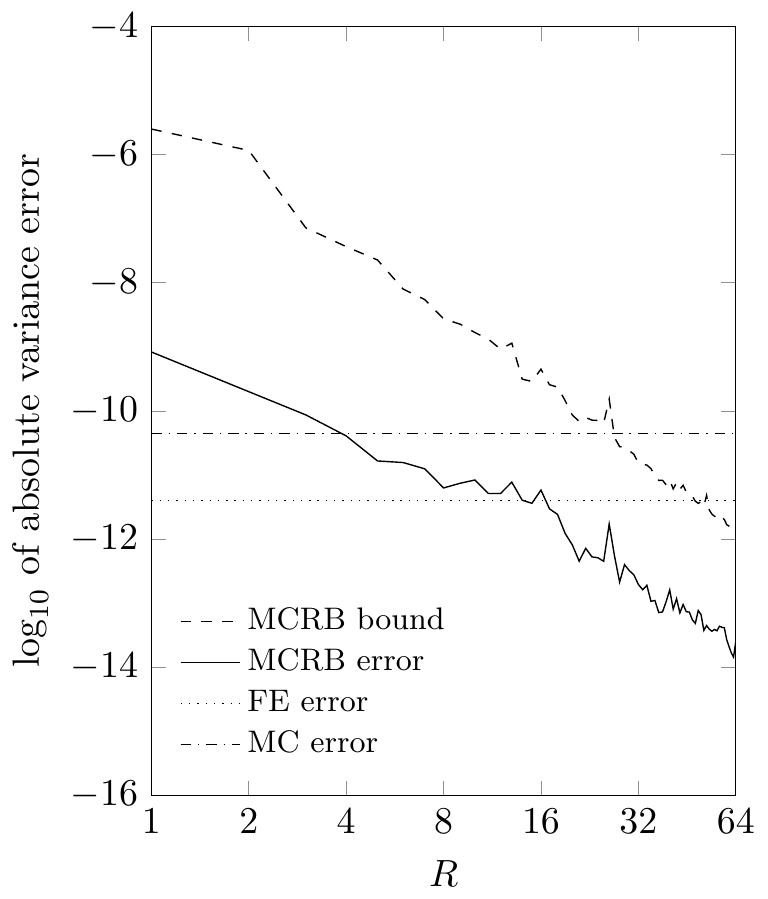}%
      \includegraphics{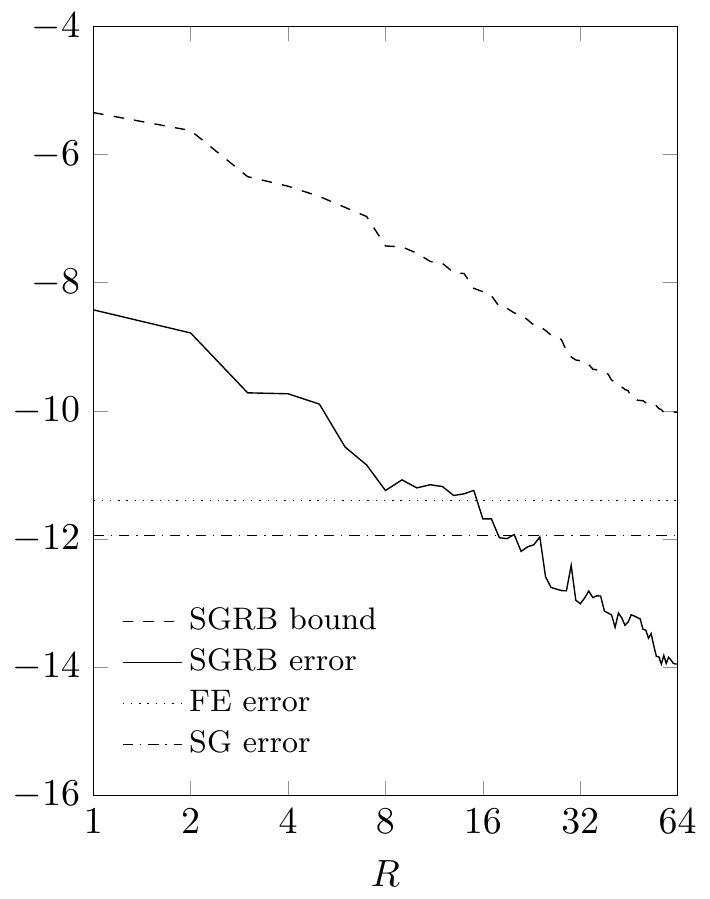}
    \else
      \begin{tikzpicture}
        \begin{myaxis}[
          ymin = 10^-13,
          ymax = 10^-3,
          xlabel = {},
          ylabel = {$\log_{10}$ of absolute expectation error},
          legend entries={
            MCRB bound,
            MCRB error,
            FE error,
            MC error},
          title = {Monte Carlo}]
          \addplot table [x=R,y=Bound]{expectationMonteCarloL2.dat};
          \addplot table [x=R,y=Error]{expectationMonteCarloL2.dat};
          \addplot table [x=R,y=Error]{expectationFiniteElementReferenceL2.dat};
          \addplot table [x=R,y=Error]{expectationMonteCarloReferenceL2.dat};
        \end{myaxis}
      \end{tikzpicture}
      \begin{tikzpicture}
        \begin{myaxis}[
          ymin = 10^-13,
          ymax = 10^-3,
          xlabel = {},
          legend entries={
            SGRB bound,
            SGRB error,
            FE error,
            SG error},
          title = {Stochastic Galerkin}]
          \addplot table [x=R,y=Bound]{expectationStochasticGalerkinL2.dat};
          \addplot table [x=R,y=Error]{expectationStochasticGalerkinL2.dat};
          \addplot table [x=R,y=Error]{expectationFiniteElementReferenceL2.dat};
          \addplot table [x=R,y=Error]{expectationStochasticGalerkinReferenceL2.dat};
        \end{myaxis}
      \end{tikzpicture}
      \begin{tikzpicture}
        \begin{myaxis}[
          ymin = 10^-16,
          ymax = 10^-4,
          ylabel = {$\log_{10}$ of absolute variance error},
          legend entries={
            MCRB bound,
            MCRB error,
            FE error,
            MC error}]
          \addplot table [x=R,y=Bound]{varianceMonteCarloL2.dat};
          \addplot table [x=R,y=Error]{varianceMonteCarloL2.dat};
          \addplot table [x=R,y=Error]{varianceFiniteElementReferenceL2.dat};
          \addplot table [x=R,y=Error]{varianceMonteCarloReferenceL2.dat};
        \end{myaxis}
        \end{tikzpicture}
        \begin{tikzpicture}
        \begin{myaxis}[
          ymin = 10^-16,
          ymax = 10^-4,
          legend entries={
            SGRB bound,
            SGRB error,
            FE error,
            SG error}]
          \addplot table [x=R,y=Bound]{varianceStochasticGalerkinL2.dat};
          \addplot table [x=R,y=Error]{varianceStochasticGalerkinL2.dat};
          \addplot table [x=R,y=Error]{varianceFiniteElementReferenceL2.dat};
          \addplot table [x=R,y=Error]{varianceStochasticGalerkinReferenceL2.dat};
        \end{myaxis}
      \end{tikzpicture}
    \fi
  \end{center}
  \caption{Log-log plots of the errors in the approximation of the expectation (top row) and variance (bottom row) of a linear functional with an MCRB method (left column) and an SGRB method (right column) depending on the dimension $\nRb$ of the reduced spaces, measured in terms of an approximate $\lebesgue{}{\dPara}$-norm. Respective error bounds and approximate FE/SG/MC discretization errors. It is a coincidence that the approximate FE and MC errors of the estimated expectation are very close in this outcome of the random experiment.\label{fig:convergenceL2}}
\end{figure}

We assess the convergence globally over $\dPara$ in order to confirm that the point-wise observation in the deterministic parameter space provided by \Cref{fig:convergencePointwise} is not a lucky coincidence. To this end, we employ an $\lebesgue{}{\dPara}$-norm, approximated using Monte Carlo quadrature with $\nTest=64$ samples shown as circles in \Cref{fig:parameterDependence}. The convergence results are presented in \Cref{fig:convergenceL2}. Since we have averaged over the parameter space, the plots appear less random than the plots in \Cref{fig:convergencePointwise}. The convergence of the estimates and the corresponding bounds correspond quite well. Moreover, the MCRB and SGRB methods perform similar in terms of accuracy per number of basis functions.

In \Cref{fig:convergencePointwise,fig:convergenceL2}, it appears that the SGRB error bound over-estimates the actual error more severely (by 4 orders of magnitude) than the MCRB error bound (2 orders of magnitude). A closer inspection of the individual components of the error estimate reveals that for larger $\nRb$ the lower-order term involving the continuity factor becomes responsible for the major portion of the error estimate. In particular, for $\nRb=64$ at the parameter point corresponding to \Cref{fig:convergencePointwise}, the terms on the right-hand side of \cref{theorem:varianceBoundsWeak} amount to approximately $4.2\cdot10^{-12}$, $1.2\cdot10^{-19}$ and $5.2\cdot10^{-14}$, respectively.

\section{Conclusion}

We have observed that the SGRB method can deliver estimates of the expectation and variance of linear outputs with an 
accuracy similar to the MCRB method. Also, the SGRB error bounds regarding the expected value were very close to the 
respective MCRB bounds in our experiments. Concerning the variance, the presented SGRB bounds overestimate the error 
more severely than the available MCRB bounds, which opens opportunities for future improvement of the SGRB variance 
bound. Nevertheless, the MCRB and SGRB variance bounds both converge at the same order depending on the number of 
reduced basis functions. This behavior is reflected by the theory, which predicts the same order of convergence in terms 
of dual norms of residuals.

The MCRB statistical output estimates and error bounds require a Monte Carlo sampling of the reduced quantities 
point-wise in the random parameter domain. In our tests, 1024 samples were sufficient to balance the finite element 
error for the expectation, but an accurate prediction of the variance would require even more samples.~The SGRB 
estimates and bounds, on the other hand, are obtained by an exact integration of the corresponding reduced basis 
expansions in the setup phase of the reduced-order model, and, thus, do not rely on Monte Carlo sampling. As a 
consequence, the primal and dual SGRB problems need to be solved only once for each new deterministic parameter. 
This benefit comes at the cost of a more expensive offline phase. In our tests, the SGRB and the MCRB 
methods achieved a similar reduction of degrees of freedom for a given error tolerance. As a consequence, the possible 
online speedup of SGRB methods compared to MCRB methods is in the order of magnitude of the number of 
Monte Carlo samples. This is particularly attractive in scenarios where evaluating the reduced order 
model shall be as inexpensive as possible for a large number of parameter queries while the offline costs are not a 
primary concern.




\bibliographystyle{elsarticle-num-names} 
\bibliography{input/literature.bib}
%
%
%
\end{document}